\documentclass[11pt]{amsart}

\usepackage[margin=1.25in]{geometry}
\usepackage{amscd,amssymb, amsmath, setspace, graphicx,color,mathtools}
\usepackage{url}
\usepackage{hyperref}
\usepackage[utf8]{inputenc} 
\usepackage{float}
\usepackage{verbatim}

\makeatletter
\@namedef{subjclassname@1991}{2020 Mathematics Subject Classification}
\makeatother

\DeclareMathOperator{\diag}{diag}
\DeclareMathOperator{\spn}{span}
\DeclareMathOperator{\trace}{tr}
\DeclareMathOperator{\rlct}{RLCT}
\DeclareMathOperator{\rank}{rank}
\DeclareMathOperator{\jac}{Jac}
\DeclareMathOperator{\conv}{conv}

\numberwithin{equation}{section}

\theoremstyle{plain}
\newtheorem{theorem}{Theorem}[section]
\newtheorem{prop}[theorem]{Proposition}

\newtheorem{cor}[theorem]{Corollary} 
\newtheorem{conj}[theorem]{Conjecture}
\newtheorem{lemma}[theorem]{Lemma}

\newtheorem{fact}[theorem]{Fact}

\theoremstyle{definition}
 
\newtheorem{defn}[theorem]{Definition}
\newtheorem{remark}[theorem]{Remark}

\newtheorem*{ex*}{Example}

\newtheorem{example}[theorem]{Example}

\newcommand{\A}{\mathbb{A}}
\newcommand{\Z}{\mathbb{Z}}
\newcommand{\C}{\mathbb{C}}
\newcommand{\R}{\mathbb{R}}

\newcommand{\bdelta}{\boldsymbol\delta}

\newcommand{\cJ}{\mathcal{J}}

\title[Singular Learning Theory for Factor Analysis]{Singular Learning Theory for Factor Analysis} 
\author[M. Drton]{Mathias Drton}
\address{Department of Mathematics, TUM School of Computation, Information and Technology, Technical University of Munich, Germany}
\email{mathias.drton@tum.de}
\author[E. Gross]{Elizabeth Gross}
\address{Department of Mathematics, University of Hawaii at M\={a}noa, U.S.A.}
\email{egross@math.hawaii.edu}
\author[D. Kosta]{Dimitra Kosta}
\address{School of Mathematics, University of Edinburgh, UK}
\email{d.kosta@ed.ac.uk}
\author[A. Leykin]{Anton Leykin}
\address{School of Mathematics, Georgia Institute of Technology, U.S.A.}
\email{leykin@math.gatech.edu}
\author[A. McCormack]{Andrew McCormack}
\address{Department of Mathematical and Statistical Sciences, University of Alberta, Canada}
\email{mccorma2@ualberta.ca}
\author[S. Sullivant]{Seth Sullivant}
\address{Department of Mathematics, North Carolina State University, U.S.A.}
\email{smsulli2@ncsu.edu}
\author[D. Windisch]{Daniel Windisch}
\address{Department of Computer Science, KU Leuven, Belgium}
\email{daniel.windisch.math@gmail.com}

\begin{document}

\begin{abstract}
Watanabe's singular learning theory provides a framework for asymptotic analysis of Bayesian model selection for statistical models with singularities, where traditional statistical regularity assumptions fail.  Learning coefficients, also known as real log canonical thresholds, play a central role in singular learning, as they govern the asymptotic behavior of Bayesian marginal likelihood integrals in settings where the Laplace approximations used for regular statistical models are not applicable.  Learning coefficients are algebraic invariants that quantify the geometric complexity of a model and reveal how the singular structure impacts the model's generalization properties.  In this paper, we apply algebraic methods to study the learning coefficients of factor analysis models, which are widely used latent variable models for continuously distributed data.  Our main result provides exact formulas for learning coefficients of factor analysis models. Moreover, we study the singularity types of specific factor analysis models in detail.
%
%
%
\end{abstract}


\maketitle

\section{Introduction}\label{section:introduction}

A Bayesian approach to statistical model selection involves the evaluation of the \emph{marginal likelihood}, which is obtained by integrating the likelihood function against the prior distribution \cite[Chap.~7]{robert:2007}.  In particular, the posterior probability of a considered model equals the normalization of the marginal likelihood under weighting by a prior model probability.
As an exact value for the marginal likelihood integral is difficult to obtain in general, different approximation schemes have been developed \cite{friel:2012}.  The theme taken up in this paper is asymptotic approximation and, specifically, the widely used approach of \emph{Bayesian information criteria} that originated in the work of \cite{Schwarz:1978}.  Bayesian information criteria (BIC) provide a proxy for the logarithm of the marginal likelihood that is formed by penalizing the maximum log-likelihood achievable in a model.  The BIC penalty for a $d$-dimensional statistical model equals $\frac{d}{2} \log(n)$, where $n$ is the sample size.  Under regularity conditions, this penalty coincides with the Laplace approximation of the log-marginal likelihood up to a remainder that is bounded in probability \cite{Haughton:1988}.

A \emph{singularity} of a parametric statistical model is a parameter vector at which the model's Fisher-information matrix is singular.  At singularities, the BIC approximation need no longer hold.  However, seminal work of Watanabe shows, under mild analyticity assumptions on the model and its parametrization, that even at singularities, the logarithm of the marginal likelihood integral can be approximated with the help of penalties of the form $\ell \log(n)$, where $\ell$ is a geometric invariant referred to as a \emph{learning coefficient} \cite{watanabe:book}.  Watanabe's work uses  resolution of singularities of real algebraic and analytic varieties to build a stochastic version of the theory of singular Laplace integrals \cite{arnold-vol2:1988}.  Geometrically, the learning coefficient at a given parameter vector is the \emph{real log canonical threshold} of a function defining the real fiber of the model's parametrization at the considered point.  Although computing a resolution of singularities is infeasible at the scale of practical statistical problems, learning coefficients have been derived for several families of models.  This information can then be exploited for statistical model selection via the \emph{singular BIC (sBIC)} methodology proposed in \cite{sbic:2017}.

The  prime source of singular statistical models is latent variable models such as mixture models \cite{rousseau:2011,yamazaki-watanabe:2003,aoyagi:2010:commstat}, reduced rank regression models \cite{aoyagi-watanabe:2005}, neural networks \cite{aoyagi:2009}, or latent tree models \cite{zwiernik:2011,latenttrees:2017}.  A key example of a latent variable model that has not yet been studied from the perspective of Watanabe's singular learning theory is the factor analysis model, which is arguably the most fundamental model with continuous latent variables \cite{harman:1976}.
In this paper, we seek to address this gap and initiate the study of learning coefficients for factor analysis.

Factor analysis models explain dependence among a number of observed random variables by a smaller number of latent random variables, called factors. 
Suppose we observe a sample of independent and identically distributed (i.i.d.) random vectors $X_1,\dots,X_n$, each vector being $p$-dimensional and, without loss of generality, centered to have a zero mean vector.
In the model with $k$ factors, we assume that the observed random vectors are generated as
\[
X_i = \Lambda F_i + \epsilon_i, \quad i=1,\dots,n,
\]
where $F_i$ is a $k$-vector comprised of i.i.d.~standard normal random variables, the matrix $\Lambda$ is a real-valued matrix in $\mathbb{R}^{p\times k}$ referred to as the \emph{factor loading matrix}, and $\epsilon_i$ is a $p$-vector that constitutes noise and has independent coordinates with $j$-th entry $\epsilon_{ij}$ normally distributed with mean zero and variance $\psi_j>0$.  It follows that the observations $X_1,\dots,X_n$ are multivariate normal, each with a covariance matrix that is parameterized as 
\begin{equation}
\label{eq:parametrization}
\Sigma_k(\psi,\Lambda) = \diag(\psi) + \Lambda\Lambda^T
\end{equation}
for a parameter vector $\psi=(\psi_1,\dots,\psi_p)\in
 \R_{ > 0}^p=(0,\infty)^p$ and a parameter matrix
$\Lambda=(\lambda_{ij})\in\mathbb{R}^{p\times k}$.  The factor analysis model can be identified with its set of covariance matrices, that is, the image of the parametrization map $\Sigma_k$.  The dimension of this image is the minimum of
\begin{equation}
  \label{eq:dimension}
  d_k\;\coloneqq\;\dim\big( \Sigma_k\big( \R_{ > 0}^p,\mathbb{R}^{p\times k} \big)\big)
  \;=\;
  (k+1)p - \binom{k}{2},
\end{equation}
and the dimension $p(p+1)/2$ of the ambient space of symmetric $(p \times p)$-matrices, compare~\cite{drton:2007}. 

At regular points, the model dimension yields the BIC-penalty term $\frac{d_k}{2}\log(n)$.  However, the factor analysis model also has singularities that are important for the problem of selecting the number of factors $k$ \cite{drton:2009,sbic:2017}.  The singularities correspond to the points $(\psi,\Lambda)$ at which the rank of the Jacobian of $\Sigma_k(\psi,\Lambda)$ drops. Singularities of the factor analysis model are studied in \cite[Theorem 5.9]{anderson1956}, which gives a general sufficient condition for a point being singular. In addition, a more explicit description of the singular points of the one factor model is given in \cite{drton:2009,li:2012}.

The main theorem of this paper gives a formula for the learning coefficients of the factor analysis model.  The learning coefficients  depend on the covariance matrix $\Sigma_0$ that defines the distribution of the observations $X_1,\dots,X_n$ as well as a prior distribution on the parameters $(\psi,\Lambda)$.  We will consider the default case where the prior is a smooth and everywhere positive function, in which case the precise form of the prior has no further effect on learning coefficients. Our results on learning coefficients can then be applied to model selection among factor analysis models via the sBIC. This has been implemented in the statistical programming environment \emph{R}~\cite{sBIC-R} with the necessary information from our main Theorem~\ref{thm:main-formula-intro} below. 

\begin{theorem}
\label{thm:main-formula-intro}
     Let $r\le k$ be any two non-negative integers. The learning coefficient $\ell_{kr}$ and its order $m_{kr}$ for the factor analysis model with $k$ latent factors at a fixed generic covariance matrix $\Sigma_0 \in \R^{p\times p}$ in the $r$-factor model satisfy 
  \[
  \ell_k(\Sigma_0) = \ell_{kr} = \frac{p(k+2)+r(p-k+1)}{4} \quad
  \text{ and } \quad
    m_k(\Sigma_0) = m_{kr} = 1.
  \]
  These are exactly the invariants used in the sBIC for model selection among factor analysis models, see Definition~\ref{def:learn-coeff} and Section~\ref{subsection:sbic}.
\end{theorem}

Note that, in the considered setup, the covariance matrix $\Sigma_0$ belongs to the factor analysis model with $r$ factors.  Obtaining the exact value of the learning coefficient depending on $r$ is thus crucial for improved statistical model selection via the criterion from \cite{sbic:2017}.

In a full analysis of the case $k=1$, we show that the genericity assumption on $\Sigma_0$ is necessary and that there exist special singularities at which the learning coefficient is smaller than the value from Theorem~\ref{thm:main-formula-intro}. 


This manuscript is organized as follows: 
Section~\ref{section:preliminaries-rlct} contains preliminaries on real log canonical thresholds (aka learning coefficients) which form a general framework for the geometric theory of singular statistical models.
Section~\ref{section:preliminaries-factor-analysis} reviews factor analysis models through the lens of singular learning theory. In Section~\ref{section:learning-coefficients}, we prove Theorem~\ref{thm:main-formula-intro} after giving an alternate proof for an upper bound on the learning coefficients. We include this as we believe that this could be relevant to obtaining upper bounds for other models. Finally, we  classify the singularity types of the one-factor model (Section~\ref{subsection:distinct-one-factor-singularities}) and treat the case of generic one-factor covariance matrices (Section~\ref{subsection:one-factor}), where we analyze the exact form of the genericity assumption.

\section{Preliminaries on real log canonical thresholds}\label{section:preliminaries-rlct}

We recall some facts on the algebraic theory of learning coefficients from Shaowei Lin's paper~\cite{Lin2017Ideal}.
 Throughout, let $\Omega \subseteq \mathbb R^d$ be a compact set, and let $\mathcal A(\Omega)$ be the ring of real-valued functions that are analytic at every point in $\Omega$.

\subsection{Real log canonical thresholds of ideals}\label{subsection:rlct-ideals}

 Let $\mathcal I = \langle f_1, \ldots, f_r\rangle \subseteq \mathcal A(\Omega) $ be an ideal generated by functions $f_i$ not identically $0$, and let $\phi: \mathbb R^d \to \mathbb R$ be a smooth, nearly analytic function, that is, $\phi$ is the (point-wise) product of a function in $\mathcal{A}(\Omega)$ and a smooth function that is strictly positive on $\Omega$.  Then the zeta function
$$\zeta (z) = \int_{\Omega} (f_1(\omega)^2 + \ldots + f_r(\omega)^2)^{-z/2} |\phi(\omega)| \ d \omega$$ 
 has an analytic continuation to the whole complex plane and the poles of the continuation are positive rational numbers, see, for instance,~\cite{Atiyah:1970}. Let $\ell_\Omega(\mathcal{I};\phi)$ be the smallest of these poles and $m_\Omega(\mathcal{I};\phi)$ its multiplicity as a pole of $\zeta(z)$. The pair
 \[
 \rlct_{\Omega}(\mathcal I; \phi) = (\ell_\Omega(\mathcal{I};\phi),m_\Omega(\mathcal{I};\phi))
 \]
 is called the \emph{real log canonical threshold} of the ideal $\mathcal I$ with respect to the \emph{amplitude function} $\phi$ over $\Omega$. With slight abuse of notation, we will also refer to $\ell_\Omega(\mathcal{I};\phi)$ as the \emph{real log canonical threshold} and to $m_\Omega(\mathcal{I};\phi)$ as its \emph{multiplicity} or \emph{order}. The pair $\rlct_\Omega(\mathcal{I};\phi)$ is independent of the choice of generators $f_1,\ldots,f_r$~\cite[Proposition 6]{Lin2017Ideal}. Moreover, it can be computed locally, in the following sense. We can order pairs $(\lambda,m),(\lambda',m') \in \R^2$ by the total ordering
 \begin{equation}\label{eq:order}
 (\lambda,m) \leq (\lambda',m') \text{ if and only if } (\lambda < \lambda' \text{ or } (\lambda = \lambda' \text{ and } m \geq m')),
 \end{equation}
that is, by lexicographic order with reversed order in the second component. The following is~\cite[Proposition 4]{Lin2017Ideal} with $f(\omega) = (f_1(\omega)^2 + \ldots + f_r(\omega)^2)^{-1/2}$.

\begin{fact}[Locality of the $\rlct$]\label{fact:local-rlct}
    For every \( x \in \Omega \), there exists a compact neighborhood \( \Omega_x \subseteq \Omega \) of \( x \) such that
\[
\rlct_{x}(\mathcal{I}; \phi) \coloneqq \rlct_{\Omega_x}(\mathcal{I}; \phi) = \rlct_U(\mathcal{I}; \phi)
\]
for all compact neighborhoods \( U \subseteq \Omega_x \) of \( x \). Moreover,
\[
\rlct_{\Omega}(\mathcal{I}; \phi) = \min_{x \in \Omega} \rlct_{x}(\mathcal{I}; \phi),
\]
where the minimum is with respect to the order of~(\ref{eq:order}) and it suffices to take the minimum over all \( x \) in the analytic variety \( \mathcal{V}(\mathcal{I}) = \{ \omega \in \Omega \mid \forall g \in \mathcal{I} \ g(\omega) = 0 \} \).
\end{fact}

For homogeneous ideals, we can directly use this local notion to determine the real log canonical threshold, see~\cite[Theorem 2]{aoyagi:2013:entropy}. Recall that a function $f$ on $\Omega$ is homogeneous in the subset $\omega_1,\ldots,\omega_j$ of the variables if there exists a non-negative integer $\delta$ such that $f(a\omega_1,\ldots,a\omega_j,\omega_{j+1},\ldots,\omega_d) = a^\delta f(\omega)$ for all $\omega = (\omega_1,\ldots,\omega_d) \in \Omega$.

\begin{fact}[Homogenous ideals]\label{fact:homogeneous}
    Let $1 \leq j \leq d$ and suppose that the space $U = \{\omega \in \Omega \mid \omega_1 = 0, \ldots, \omega_j = 0\}$ is non-empty. Fix $\omega_0 \in U$. Moreover, assume that generators $f_1,\ldots,f_r$ of $\mathcal{I}$ and $\phi$ are homogeneous functions of $\omega_1,\ldots,\omega_j$, and that $\phi(\omega_0) \geq c\phi(\omega)$ for a constant $c > 0$ and all $\omega$ in a neighborhood of $\omega_0$. Then
    \[
    \rlct_{\omega_0}(\mathcal{I};\phi) \leq \rlct_{\omega}(\mathcal{I};\phi)
    \]
    for all $\omega \in \Omega$ with $\omega_i = \omega_{0i}$ for all $i > j$.
    In particular,
    \[
    \rlct_\Omega(\mathcal{I};\phi) = \min_{\omega \in U} \rlct_{\omega}(\mathcal{I};\phi).
    \]
\end{fact}

\subsection{Newton polyhedra and monomial ideals}\label{subsection:newton-polyhedra}
We will now revisit how to compute RLCTs of monomial ideals via Newton polyhedra as described, for instance, in~\cite{Lin2017Ideal}; see also~\cite[Chapter 8]{arnold-vol2:1988}. Keeping with the above notation, suppose that $0 \in \Omega$ and that $\mathcal{I} = \langle \omega^{a_1},\ldots, \omega^{a_r}\rangle \subseteq \mathcal{A}(\Omega)$ is an ideal generated by monomials, where $\omega^b = \omega_1^{b_1} \cdots \omega_d^{b_d}$ for $b \in \Z_{\geq 0}^d$. As $\mathcal{I}$ is homogeneous, we know by Fact~\ref{fact:homogeneous} that, concerning the RLCT, the only point of interest is the origin. 

A combinatorial method to determine $\rlct_0(\mathcal{I};\omega^\tau)$ for a fixed $\tau = (\tau_1,\ldots,\tau_d) \in \Z_{\geq 0}^d$ is given as follows. The \textit{Newton polyhedron} of $\mathcal{I}$ is the convex polyhedron
\[
\mathcal{P}(\mathcal{I}) = \conv \{a_i + \xi \mid i \in [r], \xi \in \R^d_{\geq 0} \}.
\]
The \textit{$\tau$-distance} $\delta_\tau$ of $\mathcal{I}$ is the smallest $t \in \R_{\geq 0}$ such that $t \cdot(\tau_1+1,\ldots,\tau_d+1) \in \mathcal{P}(\mathcal{I}) $ and its \textit{$\tau$-multiplicity} $\mu_\tau$ is the codimension of the face $F$ of $\mathcal{P}(\mathcal{I})$ with $t \cdot(\tau_1+1,\ldots,\tau_d+1) \in F$.\\

\begin{fact}~\cite[Theorem 3]{Lin2017Ideal}\label{fact:monomial}
   The real log canonical threshold at the origin of a monomial ideal is given by the $\tau$-distance and $\tau$-multiplicity as 
   \[
   \rlct_0(\langle \omega^{a_1},\ldots,\omega^{a_r}\rangle; \omega^\tau) = (1/\delta_\tau, \mu_\tau).
   \]
\end{fact}

\begin{example}\label{example:linear-space}
    We compute the real log canonical threshold of a linear subspace of codimension $c$ in $\R^d$ as given by the monomial ideal $\mathcal{I} = \langle \omega_1,\ldots,\omega_c\rangle$. As the phase function, we use $1 = \omega^0$, so that $\tau = 0$. The Newton polyhedron of $\mathcal{I}$ is 
    \[
    \mathcal{P}(\mathcal{I}) = \conv \{e_i + \xi \mid i \in [c], \xi \in \R^d_{\geq 0}\} \subseteq \R^d,
    \]
    where $e_i$ denotes the $i$-th unit vector. Clearly, all points $a \in \mathcal{P}(\mathcal{I})$ satisfy $\sum_{i = 1}^c a_i \geq 1$.
    Therefore, the set $F = \{ (a_1,\ldots,a_d) \in \R^d_{\geq 0} \mid \sum_{i = 1}^c a_i = 1 \}$ forms a facet of $\mathcal{P}(\mathcal{I})$ and a point of the form $(t,\ldots,t)$ lies in $F$ if and only if $t = 1/c$, and this point is contained in the relative interior of $F$. Hence, $\delta_0 = 1/c$ and $\mu_0 = 1$ which implies $\rlct_0(\langle \omega_1,\ldots,\omega_c\rangle;1) = (c,1)$ by Fact~\ref{fact:monomial}.

\end{example}

\subsection{Calculation rules}\label{subsection:chain-sum-product-rules}
We now gather further techniques to manipulate ideals and RLCTs.
The first result, referred to as ``chain rule'' in~\cite{lin:phd} is central to many arguments in this paper.
The following statement that we will use without further reference is a slight generalization of~\cite[Proposition 8]{Lin2017Ideal} but the proof is exactly the same.

\begin{fact}[Chain rule]\label{fact:chain-rule}
     Let $\Omega \subseteq \R^d$  be open
     and $W \subseteq \Omega$ a compact semianalytic neighborhood of a point $x\in\Omega$. Let $\mathcal{I}=\langle f_1,\ldots,f_r\rangle$ be a finitely generated ideal of $\mathcal{A}(\Omega)$ and $ W = W_1 \cup \ldots \cup W_n \cup V$ a partition, where $V \subsetneqq W$ is an analytic variety and the $W_i$ are compact semianalytic subsets of dimension $d$. Let $M$ be a real analytic manifold and $\rho: M \to W$ a proper real analytic map whose restrictions $\rho^{-1}(W_i) \to W_i$ are real analytic isomorphisms, that is, bijective with real analytic inverse. Then
    \[
    \rlct_x(\mathcal{I};\phi) = \min_{y \in \rho^{-1}(x)}\rlct_y(\rho^*\mathcal{I};(\phi \circ \rho)\cdot\det\jac\rho),
    \]
    where $\rho^*\mathcal{I} = \{g \circ \rho \mid g \in \mathcal{I}\} = \langle f_1 \circ \rho, \ldots, f_r \circ \rho \rangle \subseteq \mathcal{A}(M)$ is the pullback of $\mathcal{I}$ under $\rho$ and $\jac \rho$ is the Jacobian matrix of~$\rho$.
\end{fact}

In this paper, $M$ will always be a subset of a real algebraic variety that is locally (algebraically) isomorphic to an affine space and $\rho$ will be a polynomial map.

We review two further calculation rules.

\begin{fact}[Sum and Product Rule]\cite[Proposition 7]{Lin2017Ideal}\label{fact:sum-product-rule}
    Let $\Omega_1 \subseteq \R^{d_1}$ and $\Omega_2 \subseteq \R^{d_2}$ be compact subsets, and let $\mathcal{I} \subseteq \mathcal{A}(\Omega_1)$ and $\mathcal{J} \subseteq \mathcal{A}(\Omega_2)$ be finitely generated ideals. Then, composing with the canonical projections $\Omega_1 \times \Omega_2 \to \Omega_i$, we can consider $\mathcal{I}$ and $\mathcal{J}$ as ideals of $\mathcal{A}(\Omega_1 \times \Omega_2)$. Let $\phi_1: \Omega_1 \to \R$ and $\phi_2: \Omega_2 \to \R$ be nearly analytic. Denote $\rlct_{\Omega_1 }(\mathcal{I};\phi_1) = (\ell_1,m_1)$ and $\rlct_{\Omega_2}(\mathcal{J};\phi_2) = (\ell_2,m_2)$.  Then
    \begin{enumerate}
        \item $\rlct_{\Omega_1 \times \Omega_2}(\mathcal{I} + \mathcal{J};\phi_1\cdot \phi_2) = (\ell_1+\ell_2,m_1+m_2-1)$ and
        \item $\rlct_{\Omega_1 \times \Omega_2}(\mathcal{IJ};\phi_1\cdot \phi_2) = 
        \begin{cases}
            (\ell_1,m_1) \text{ if } \ell_1 < \ell_2, \\
            (\ell_2,m_2) \text{ if } \ell_1 > \ell_2, \\
            (\ell_1,m_1 + m_2) \text{ if } \ell_1 = \ell_2.
        \end{cases}$
    \end{enumerate}
\end{fact}

\subsection{Blow-up along a linear subspace}\label{subsection:blow-up}
Blow-ups (over the real numbers) are a standard tool to construct maps $\rho$ that satisfy the conditions of the chain rule (Fact~\ref{fact:chain-rule}) and can, hence, be used for computations concerning RLCTs. 

In this paper, we will only use the blow-up along a linear subspace $L = \{x \in \R^d \mid x_1 = \ldots = x_c = 0\}$ of $\R^d$ defined by the vanishing of a subset of the coordinates. This blow-up is a real algebraic map $\rho: M \to \R^d$ from a $d$-dimensional real algebraic variety $M$. The map is an isomorphism outside of $L$ and pulls back $L$ to a real projective space $\mathbb{P}_\R^{d-1}$. The manifold $M$ can be covered by $c$ copies of $\R^d$ which we call charts and, as the RLCT is a local concept, it suffices to consider the restrictions of $\rho$ to these charts. 
For $i \in [c]$, the restriction to the $i$-th chart can be described as
\begin{align*}
        \rho_i: \R^d &\longrightarrow \R^d\\
                (x_1,\ldots,x_d) &\longmapsto (x_1x_i,\ldots,x_{i-1}x_i,x_i,x_{i+1}x_i,\ldots,x_cx_i,x_{c+1},\ldots,x_d).
\end{align*}
Its Jacobian determinant is given by $\det \jac \rho_i = x_i^{c-1}$. Note that, if $c = d$, then $L$ is just the origin in $\R^d$ in which case we call $\rho$ the blow-up at the origin.

\section{Singular Learning Theory for factor analysis}\label{section:preliminaries-factor-analysis}

In this section, we review key aspects of singular learning theory through the lens of the factor analysis model.

\subsection{Marginal likelihood and learning coefficients}

Let $X_1,\dots,X_n$ be a sample of independent and identically
distributed random vectors, with $X_i$ taking values in $\mathbb{R}^p$.  Without loss of generality, we assume the expectation
vectors $\mathbb{E}[X_i]$ to be zero.  Let
\begin{equation*}
S_n=\frac{1}{n} \sum_{i=1}^n X_iX_i^T
\end{equation*}
be the $p\times p$ sample covariance matrix, and let $\mathit{PD}_p$ denote the cone of positive definite $p\times p$ matrices.  Define the function 
\begin{equation}
  \label{eq:negloglik}
  f(\Sigma\,|\, S_n) = 
  \frac{p}{2}\log(2\pi) +\frac{1}{2}\log\det(\Sigma)
   + \frac{1}{2} 
    \trace\left( \Sigma^{-1}S_n \right), \qquad \Sigma \in\mathit{PD}_p.
\end{equation}
Then $-nf(\Sigma\,|\, S_n)$ is the Gaussian log-likelihood function, which maps a matrix $\Sigma\in\mathit{PD}_p$ to the logarithm of the joint density of $(X_1,\dots,X_n)$ when the $X_i$ are i.i.d.~multivariate normal with covariance matrix $\Sigma$.  

In the factor analysis model with $k$ factors the covariance matrix of the observations is given by the parametrization map from \eqref{eq:parametrization}, so $\Sigma=
\Sigma_k(\psi,\Lambda) = \diag(\psi) + \Lambda\Lambda^T$ with  $\psi=(\psi_1,\dots,\psi_p)\in
 \R_{ > 0}^p$ and 
$\Lambda=(\lambda_{ij})\in\mathbb{R}^{p\times k}$.  Suppose that, for a Bayesian treatment, we have chosen a prior
distribution for $(\psi,\Lambda)$ and that this prior distribution has Lebesgue density $\varphi_k(\psi,\Lambda)$ on
$ \R_{ > 0}^p\times \mathbb{R}^{p\times k}$.  In this paper we assume that the prior density $\varphi_k$ is everywhere positive, bounded, and smooth; compare, e.g., \cite{leung:2016}. 
The marginal likelihood of the $k$-factor model is now the integral
\begin{equation}
\label{eq:marginal-likelihood}
  L_{k,n}(S_n) = \\
  \int_{\mathbb{R}^{p\times k}}\int_{ \R_{ > 0}^p} 
e^{-n f\left(\Sigma_k(\psi,\Lambda)\,|\,S_n\right)}
  \varphi_k(\psi,\Lambda) \;  d\psi\,d\Lambda.
\end{equation}
While exact values of $L_{k,n}(S_n)$ are challenging to obtain, general results
from \cite[Chap.~6]{watanabe:book} describe the asymptotics of $L_{k,n}(S_n)$ as $n$ increases. 
Let 
$\Sigma_0=\Sigma_k(\psi_0,\Lambda_0)$ with $\psi_0\in \R_{ > 0}^p$
and $\Lambda_0\in\mathbb{R}^{p\times k}$ be the true covariance matrix of the multivariate normal observations $X_1,\dots,X_n$.
The negated and scaled log-likelihood function in \eqref{eq:negloglik} is minimized uniquely by $\Sigma=S_n$, and its minimal value is
\begin{equation}
\label{eq:ell-minimum}
f(S_n) := f(S_n\,|\,S_n) = \frac{p}{2}\log(2\pi)
+\frac{1}{2}\log\det(S_n) + \frac{p}{2}.
\end{equation}
From \cite[Chap.~6]{watanabe:book}, it follows that the marginal likelihood sequence obtained by varying $n$ satisfies
\begin{equation}
  \label{eq:watanabe}
  -\log L_{k,n}(S_n) = n f(\Sigma_0 | S_n) + \ell_k(\Sigma_0)
  \log(n) - \left[m_k(\Sigma_0)-1\right] \log\log(n) + O_{prob}(1),
\end{equation}
where $O_{prob}(1)$ denotes a sequence of random variables that is bounded
in probability.  
In \eqref{eq:watanabe}, $\ell_k(\Sigma_0)$ is a positive rational number and $m_k(\Sigma_0)$ is an integer, with the following terminology.

\begin{defn}
\label{def:learn-coeff}
  The coefficient $\ell_k(\Sigma_0)$ in (\ref{eq:watanabe}) is 
  the \emph{learning coefficient} of the $k$-factor model at the covariance matrix $\Sigma_0$, and $m_k(\Sigma_0)$ is its
  \emph{order} or \emph{multiplicity}.
\end{defn}

 The reader not so familiar with learning coefficients might take Theorem~\ref{fact:reduction} below (using the notation from Sections~\ref{subsection:rlct-ideals} and~\ref{subsection:fiber-ideals}) as a definition of $\ell_k(\Sigma_0)$ and $m_k(\Sigma_0)$.

As we revisit in Lemma~\ref{lemma:model-dimension}, it holds that $\ell_k(\Sigma_0)\in(0,d_k/2]$, see~\cite[Theorem 7.2]{watanabe:book}.  For the order, it holds by definition that  $m_k(\Sigma_0)\in [d_k] \coloneqq \{1,\dots,d_k\}$, see~\cite[p. 32]{watanabe:book}.  Here, $d_k$ is the dimension from \eqref{eq:dimension}.   Our notation suppresses any dependence of the learning coefficient and its order on the prior density $\varphi_k$.  Indeed, the assumed smoothness and positivity ensures that $\varphi_k$ is bounded above and bounded away from zero on every compact subset, which implies that $\ell_k(\Sigma_0)$ and $m_k(\Sigma_0)$ do not depend on the specific form of the prior~\cite[Lemma 3.8]{lin:phd}.

\subsection{Setup towards sBIC}\label{subsection:sbic}
The factor analysis model $\mathcal{M}_r$ with $r \leq k$ latent factors can be identified with a submodel of the $k$-factor model $\mathcal{M}_k$ as follows: As above, identify $\mathcal{M}_k$ with its space of covariance matrices $ \Sigma_k(\psi,\Lambda) = \Lambda \Lambda^T + \diag(\psi)$, where we vary $\Lambda \in \R^{p\times k}$ and $\psi \in \R_{ > 0}^p$. Now, the following conditions are equivalent for a covariance matrix $\Sigma_0 \in \mathcal{M}_k$, and, if they are satisfied, $\Sigma_0$ lies in the $r$-factor model $\mathcal{M}_r$:
\begin{enumerate}
    \item $\Sigma_0 = \Sigma_k(\psi,\Lambda)$ for some $\Lambda \in \R^{p\times k}$ of rank at most $r$ and some $\psi \in  \R_{ > 0}^p$;
    \item $\Sigma_0 = \Sigma_r(\psi,\Lambda) = \Sigma_k(\psi, [\Lambda \ 0])$ for some $\Lambda \in \R^{p\times r}$ and some $\psi \in  \R_{ > 0}^p$.
\end{enumerate}

 So, we have a (totally) ordered system of statistical models
\[
\mathcal{M}_0 \subset \mathcal{M}_1 \subset \ldots \subset \mathcal{M}_k,
\]
and hence have a marginal likelihood estimation and model selection problem where the \emph{singular Bayesian Information Criterion} (sBIC) of Drton and Plummer can be applied, see~\cite{sbic:2017}. For all choices $r \leq k$, there exists a Zariski open subset $U_{kr}$ of the space of $(p \times k)$-matrices of rank at most $r$ such that, for all $\Sigma_1 = \Sigma_k(\psi_1,\Lambda_1), \Sigma_2 = \Sigma_k(\psi_2,\Lambda_2) \in \mathcal{M}_r$ with $\Lambda_1,\Lambda_2 \in U_{kr}$, we have $\ell_k(\Sigma_1) = \ell_k(\Sigma_2)$ and $m_k(\Sigma_1) = m_k(\Sigma_2)$. We denote this (\textit{generic}) learning coefficient and order of $\mathcal{M}_k$ along $\mathcal{M}_r$ by $\ell_{kr}$ and $m_{kr}$, respectively, that is, $\ell_{kr} = \ell_k(\Sigma_0)$ and $m_{kr} = m_k(\Sigma_0)$, where $\Sigma_0 = \Lambda\Lambda^T + \diag(\psi) \in \mathcal{M}_r$ with $\Lambda \in U_{kr}$. Note that the definition of $\ell_{kr}$ and $m_{kr}$ is independent of the data $X_1,\ldots,X_n$. In order to apply the sBIC to model selection among the models $\mathcal{M}_0,\ldots,\mathcal{M}_k$, it suffices to compute the numbers $\ell_{sr}$ and $m_{sr}$ for all $r,s \in  \{0,\ldots,k\}$ with $r \leq s$. 

\subsection{Fiber ideals for factor analysis}\label{subsection:fiber-ideals}
The \textit{fiber ideal} of the $k$-factor model with $p$-dimensional observations at the covariance matrix $\Sigma_0 = (\sigma_{ij})$ is defined as
\[
\overline{\mathcal{I}}_{p,k}(\Sigma_0) = \mathcal{I}_{p,k}(\Sigma_0) + \mathcal{I}_{p,k}'(\Sigma_0),
\]
where
\[
\mathcal{I}_{p,k}(\Sigma_0) = \langle \lambda_i\lambda_j^T - \sigma_{ij} \mid 1 \leq i < j \leq p \rangle,
\]
\[
\mathcal{I}_{p,k}'(\Sigma_0) = \langle \lambda_i\lambda_i^T + \psi_i - \sigma_{ii} \mid 1 \leq i \leq p \rangle,
\]
$\Lambda$ is a $(p\times k)$-matrix of indeterminates with row vectors $\lambda_i$, and $\psi$ is a $p$-vector of indeterminates. 
Note that the \textit{partial fiber ideal} $\mathcal{I}_{p,k}(\Sigma_0)$ is an ideal of the ring $\mathcal{A}(\R^{p\times k})$ of functions that are analytic in every point of $\R^{p\times k}$.
Using the calculation rules for real log canonical thresholds of ideals from Section~\ref{section:preliminaries-rlct} will allow us to reduce the problem of  determining $\ell_k(\Sigma_0)$ and $m_k(\Sigma_0)$ to computing real log canonical thresholds of $\mathcal{I}_{p,k}(\Sigma_0)$. 
 This is crucial to many proofs throughout the paper. Indeed, all lemmas in Section~\ref{section:learning-coefficients} and~\ref{sec:distinct-singularities} are devoted to determining or giving upper bounds for $\min_\Lambda \rlct_\Lambda(\mathcal{I}_{p,k}(\Sigma_0);1)$ from Theorem~\ref{fact:reduction} and the corresponding results on learning coefficients then follow as corollaries.

\begin{theorem}\label{fact:reduction}
    Let $\Sigma_0\in \R^{p\times p}$ be a symmetric positive definite matrix lying in the $k$-factor model $\mathcal{M}_k$. Then the following holds for the learning coefficient $\ell_k(\Sigma_0)$ and its order $m_k(\Sigma_0)$ of $\mathcal{M}_k$:
    \[
    (2\ell_k(\Sigma_0),m_k(\Sigma_0)) = \min_{(\psi,\Lambda)}\rlct_{(\psi,\Lambda)}(\overline{\mathcal{I}}_{p,k}(\Sigma_0);1) = \min_{\Lambda} \rlct_\Lambda(\mathcal{I}_{p,k}(\Sigma_0);1) + (p,0),
    \]
    where the minima range over all $\Lambda \in \R^{p \times k}$ and $\psi \in  \R_{ > 0}^p$ with $\Sigma_k(\psi,\Lambda) = \Sigma_0$.
\end{theorem}
\begin{proof}
    First, we will transform the fiber ideal $\overline{\mathcal{I}}_{p,k}(\Sigma_0)$ by the following map:
\begin{align*}
    \rho: (\psi'_1,\ldots,\psi'_p,\lambda_1,\ldots,\lambda_p) \mapsto ( \psi_1' -\lambda_1\lambda_1^T  + \sigma_{11},\ldots,\psi_p' -\lambda_p\lambda_p^T  + \sigma_{pp},\lambda_1,\ldots,\lambda_p)
\end{align*}
The Jacobian matrix of $\rho$ is an upper triangular matrix with ones along the diagonal, so $\det \jac \rho = 1$. We compute the pullback of $\overline{\mathcal{I}}_{p,k}(\Sigma_0)$ as
\[
\rho^*\overline{\mathcal{I}}_{p,k}(\Sigma_0) = \mathcal{I}_{p,k}(\Sigma_0) + \mathcal{J},
\]
where $\mathcal{J} = \langle \psi_1',\ldots,\psi_p'\rangle$. Now, let $\phi = \varphi_k$ be the prior density which, as noted before, can be assumed identical to $1$~\cite[Lemma 1]{Lin2017Ideal}. It follows from the Chain Rule (Fact~\ref{fact:chain-rule}) that 
    \[
    \rlct_x(\overline{\mathcal{I}}_{p,k}(\Sigma_0);\varphi_k) = \rlct_x(\overline{\mathcal{I}}_{p,k}(\Sigma_0);1) =\min_{y \in \rho^{-1}(x)}\rlct_y(\mathcal{I}_{p,k}(\Sigma_0) + \mathcal{J};1)
    \]
for every $x = (\psi,\Lambda) \in  \R_{ > 0}^p \times \R^{p\times k}$. Moreover, Fact~\ref{fact:homogeneous} and Example~\ref{example:linear-space} show that the number $\rlct_{(\psi_1',\ldots,\psi_p')}(\mathcal{J};1)$ takes $(2\frac{p}{2},1) = (p,1)$ as its minimal value. Note for this, that $\mathcal{J}$ is a homogeneous ideal and that $(\psi_1',\ldots,\psi_p',\Lambda) = (0,\ldots,0,\Lambda)$ is always in the fiber $\rho^{-1}(\psi,\Lambda)$ for some $\Lambda$ because the diagonal entries $\sigma_{ii}$ of the positive semi-definite matrix $\Sigma_0$ are non-negative. 
Applying Fact~\ref{fact:sum-product-rule}(1) to the sum $\mathcal{I}_{p,k}(\Sigma_0) + \mathcal{J}$, we infer
\begin{equation}\label{equation:reduction}
\rlct_{(\psi,\Lambda)}(\overline{\mathcal{I}}_{p,k}(\Sigma_0);\varphi_k) = \rlct_{\Lambda}(\mathcal{I}_{p,k}(\Sigma_0);1) + (p,0)
\end{equation}
for all $(\psi,\Lambda) \in  \R_{ > 0}^p \times \R^{p\times k}$. Combining this with~\cite[Theorem 2]{Lin2017Ideal} gives the statement of the theorem.
\end{proof}

Another calculation rule is specific for factor analysis models and deals with multiplication of the fixed covariance matrix by a diagonal matrix.

\begin{lemma}
  \label{lem:torus}
  Let $\Gamma=\diag(\gamma)$ be a diagonal matrix given by a vector
  $\gamma\in(\mathbb{R}\setminus\{0\})^p$ with all entries non-zero.
  If $\Sigma_0 \in \R^{p\times p}$ is in the $k$-factor model then so is
  $\Gamma\Sigma_0\Gamma$ and moreover
  \[
  \ell_k(\Sigma_0)=\ell_k(\Gamma\Sigma_0\Gamma) \ \text{and} \
  m_k(\Sigma_0)=m_k(\Gamma\Sigma_0\Gamma).
  \]
\end{lemma}
\begin{proof}
 As $\Gamma$ is invertible, the off-diagonal entries of $\Gamma( \Lambda\Lambda^T - \Sigma_0)\Gamma = \Gamma \Lambda\Lambda^T\Gamma - \Gamma\Sigma_0\Gamma$ generate the same ideal as the off-diagonal entries of $\Lambda \Lambda^T - \Sigma_0$ in the ring $\mathcal{A}(\R^{p\times k}$), where $\Lambda$ is a $(p\times k)$-matrix of indeterminates, namely the ideal $\mathcal{I}_{p,k}(\Sigma_0)$. The Jacobian determinant of the map $\rho: \Lambda' \mapsto \Lambda = \Gamma^{-1}\Lambda' \Gamma^{-1}$ is the inverse of a monomial in the diagonal entries of $\Gamma$ and, hence, a non-zero constant. The result follows by the chain rule (Fact~\ref{fact:chain-rule}) because $\rho^* \mathcal{I}_{p,k}(\Sigma_0) = \mathcal{I}_{p,k}(\Gamma \Sigma_0 \Gamma)$.
\end{proof}

\begin{lemma}
  \label{lemma:model-dimension}
  Consider a $k$-factor model, where $p$ is the dimension of the observations.  
  \begin{enumerate}
      \item If $d_k \leq p(p+1)/2$ and the covariance matrix $\Sigma_0$ is chosen generically from the $k$-factor model, then the learning coefficient is $\ell_{kk} = \ell_k(\Sigma_0)=d_k/2$ and the order is $m_{kk} = m_k(\Sigma_0)=1$.
      \item If $r \leq k$ such that $d_r > p(p+1)/2$ and $\Sigma_0$ is chosen generically from the $k$-factor model, then $\ell_{kr} = \ell_{kk} = \ell_k(\Sigma_0) = p(p+1)/4$ and $m_{kr} = m_{kk} = m_k(\Sigma_0) = 1$.
  \end{enumerate}
\end{lemma}

\begin{proof}
    Throughout, we will use ideas from the proof of Theorem 2 in~\cite{drton:2007}.
    First note that, by~\cite[Theorem 2]{drton:2007},  if $r \leq k$ with $d_r > p(p+1)/2$, then $d_k > p (p+1)/2$ and the images of the parametrization maps $\Sigma_r$ and $\Sigma_k$ are Zariski dense in the cone of positive definite $(p\times p)$-matrices. Therefore, $\ell_{kr} = \ell_{kk}$ and $m_{kr} = m_{kk}$ in this case.

    Now we go to the general case where $k \in \{0,\ldots, p\}$. The parametrization map $\Sigma_k$ extends to a morphism $\Sigma_k: \A_\C^p \times \A^{p\times k}_\C \to \overline{\mathcal{M}_k}$ of schemes of finite type over $\C$, where $\overline{\mathcal{M}_k}$ is the complex Zariski closure of the $k$-factor model, which has dimension $D_k \coloneqq\min\{d_k,p(p+1)/2\}$. By~\cite[Corollary 10.7]{Hartshorne} and~\cite[Theorem 10.2]{Hartshorne}, the fiber of generic $\Sigma_0 \in \overline{\mathcal{M}_k}$ is smooth over $\C$ of codimension $D_k$. Clearly, for $\Sigma_0 \in \mathcal{M}_k$ the corresponding fiber has a real point in $ \R_{ > 0}^p \times \R^{p\times k}$, just by the definition of $\mathcal{M}_k$. So, if $\Sigma_0 \in \mathcal{M}_k$ is chosen generically, this real point is a smooth point of the fiber, which implies that the dimension of the set of points of the fiber that lie in $ \R_{ > 0}^p \times \R^{p\times k}$ also has (real) codimension $D_k$~\cite[Theorem 2.2.9]{Mangolte}, and it is smooth over $\R$. 
    Consequently, using~\cite[Corollary 1.6]{Artin-analytic}, for each point in $ \R_{ > 0}^p \times \R^{p\times k}$ on the fiber, there exists a (Euclidean) neighborhood $U$ and a real analytic isomorphism $\R^p \times \R^{p\times k} \to U$ that transforms the defining ideal $\overline{\mathcal{I}}_{p,k}(\Sigma_0)$ to an ideal generated by $D_k$ variables. The statement now follows by Example~\ref{example:linear-space}, (\ref{equation:reduction}) and Theorem~\ref{fact:reduction}.
\end{proof}

\subsection{LQ-decomposition}\label{subsection:LQ-decomp} We use a specific type of transformation for $\mathcal{I}_{p,k}(\Sigma_0)$ to compute the local RLCTs from Theorem~\ref{fact:reduction}.
 Denote by $\mathcal{L}^{p\times k}_{r,+}$ the space of real $(p\times k)$-matrices of the form 
\begin{equation}\label{equ:lambda'}
    \Lambda' = \begin{bmatrix}
    \Lambda'_{11} & 0 \\
    \Lambda'_{21} & \Lambda'_{22}
\end{bmatrix}
\end{equation}
with $\Lambda'_{11}$ a lower triangular $r \times r$ matrix with positive diagonal entries.

\begin{fact}\label{fact:LQ}
    Let $\Sigma_0 \in \R^{p\times p}$ be in the $k$-factor model $\mathcal{M}_k$, and let $r \in \{0,\ldots,k\}$ be the minimum rank of any matrix $\Lambda \in \R^{p\times k}$ such that $\Sigma_0 = \diag(\psi) + \Lambda \Lambda^T$ for some $\psi \in  \R_{ > 0}^p$. Then
    \[
    \min_\Lambda \rlct_\Lambda(\mathcal{I}_{p,k}(\Sigma_0);1) = \min_{\Lambda'} \rlct_{\Lambda'}(\mathcal{I}_{p,k}(\Sigma_0);1),
    \]
     where the minima range over all $\Lambda \in \R^{p \times k}$ and $\Lambda' \in \mathcal{L}^{p\times k}_{r,+}$ respectively, with $\Sigma_k(\psi,\Lambda) = \Sigma_0 = \Sigma_k(\psi',\Lambda')$ for some $\psi, \psi' \in  \R_{ > 0}^p$, and the elements of $\mathcal{I}_{p,k}(\Sigma_0)$ on the right side of the equation are considered as functions on $\mathcal{L}^{p\times k}_{r,+}$.
\end{fact}

\begin{proof}
First, we fix $\Lambda_* \in \R^{p\times k}$ with $\Sigma_0 = \diag(\psi) + \Lambda_* \Lambda_*^T$, so that $\rank \Lambda_* \in \{r,r+1,\ldots,k\}$ and, hence, some $(r\times r)$-minor of $\Lambda_*$ is non-zero. The linear transformation that permutes the rows and columns of $\Lambda_*$ such that the non-vanishing $(r\times r)$-minor is transformed into the determinant of the matrix consisting of the first $r$ rows and columns of a $(p\times k)$-matrix has non-zero constant Jacobian determinant. So, without changing $\rlct_{\Lambda_*}(\mathcal{I}_{p,k}(\Sigma_0))$, we can assume that
\[
\Lambda_* = 
\begin{bmatrix}
    \Lambda_{*11} & \Lambda_{*12}\\
    \Lambda_{*21} & \Lambda_{*22}
\end{bmatrix},
\]
where $\Lambda_{*11} \in \R^{r\times r}$ is of full rank. Furthermore, as the non-vanishing of a minor defines a Zariski open set, we can find a (compact) neighbourhood $\Omega \subseteq \R^{p\times k}$ of $\Lambda_*$ such that 
\[
\Lambda = 
\begin{bmatrix}
    \Lambda_{11} & \Lambda_{12}\\
    \Lambda_{21} & \Lambda_{22}
\end{bmatrix}
\]
with $\Lambda_{11} \in \R^{r\times r}$ of full rank, for all $\Lambda \in \Omega$. The considerations of~\cite[Example 4.1.2]{Absil:2008} show that QR-decomposition is an analytic isomorphism of full-rank square matrices, so that we can write 
$\Lambda_{11}^T = Q_{11}^T (\Lambda'_{11})^T$ with $Q_{11} \in O(r)$, the space of real $(r\times r)$ orthogonal matrices, and $\Lambda'_{11} = (\lambda'_{ij})$ is a real $(r\times r)$ lower triangular matrix with $\lambda'_{ii} > 0$ for all $i\in \{1,\ldots,r\}$, where $Q_{11}$ and $\Lambda'_{11}$ are uniquely determined and depend analytically on $\Lambda$. After transposition, $\Lambda_{11} = \Lambda'_{11} Q_{11}$. Define $Q_{12} = (\Lambda'_{11})^{-1} \Lambda_{12} \in \R^{r \times (k-r)}$, so that 
\[
\begin{bmatrix}
    \Lambda_{11} & \Lambda_{12}
\end{bmatrix} =
\Lambda_{11}'
\begin{bmatrix}
    Q_{11} & Q_{12}
\end{bmatrix}.
\]
The matrix $\begin{bmatrix}
    Q_{11} & Q_{12}
\end{bmatrix}$ has orthonormal rows by ~\cite[Theorem A9.8]{muirhead:1982}, and note that $Q_{12}$ depends analytically on $\Lambda$. Now fix any analytic way of finding an orthonormal basis of the orthogonal complement of a given $r$-dimensional subspace of $\R^k$ and, in this way, complete to an orthogonal matrix
\[
Q = \begin{bmatrix}
    Q_{11} & Q_{12}\\
    Q_{21} & Q_{22}
\end{bmatrix}
\in O(k).
\]
Set 
\begin{equation*}\label{equ:lambda'}
    \Lambda' = \begin{bmatrix}
    \Lambda'_{11} & 0 \\
    \Lambda'_{21} & \Lambda'_{22}
\end{bmatrix}
\in \mathcal{L}^{p\times k}_{r,+}.
\end{equation*}
where $\Lambda'_{21}$ and $\Lambda'_{22}$ are uniquely determined by $\Lambda' = \Lambda Q^T = \Lambda Q^{-1}$. We found a decomposition $\Lambda = \Lambda'Q$ for $\Lambda \in \Omega$, where $\Lambda' \in \mathcal{L}^{p\times k}_{r,+}$ and $Q \in O(k)$ are uniquely determined and depend analytically on $\Lambda$. Denote this decomposition as
\begin{align*}
    \delta: \Omega \to \mathcal{L}^{p\times k}_{r,+} \times O(k),\;
    \Lambda \mapsto  (\Lambda',Q),
\end{align*}
and note the first projection of $\delta(\Omega)$ is a full dimensional set because the projection is an open map.
This $\delta$ is the analytic inverse of the real analytic isomorphism
\begin{align}\label{equ:multiplication-map}
    \mu: \delta(\Omega) \to \Omega,\;
            (\Lambda',Q) \mapsto \Lambda'Q
\end{align}
The Jacobian determinant of $\mu$ is a unit in the ring of real analytic functions by the very fact that $\mu$ is a real analytic isomorphism. As $\Lambda'(\Lambda')^T = \Lambda Q Q^T \Lambda^T = \Lambda \Lambda^T$ for each $\Lambda \in \Omega$, we see that the pullback $\mu^*\mathcal{I}_{p,k}(\Sigma_0)$ is generated by the same functions as $\mathcal{I}_{p,k}(\Sigma_0)$ itself replacing $\lambda_{ij}$ by $\lambda_{ij}'$. Using that the Jacobian determinant of $\mu$ is strictly positive on a neighbourhood of $(\Lambda',Q)$, that $Q$ is already determined by $\Lambda'$, and that $\delta(\Omega)$ is full-dimensional yields the result.
\end{proof}

\section{Learning coefficients for factor analysis models}\label{section:learning-coefficients}

\subsection{Diagonal covariance matrices}\label{subsection:diagonal}

We begin by studying the instance when the given covariance matrix $\Sigma_0$ is diagonal and therefore lies in the zero factor model. In this case, the partial fiber ideal is independent of $\Sigma_0$ and given by
\begin{equation*}
  \mathcal{I}_{p,k,0} = \mathcal{I}_{p,k}(\Sigma_0)=  \langle \lambda_i\lambda_j^T \mid 1\le i<j\le p\rangle,
\end{equation*}
see Section~\ref{subsection:fiber-ideals}. Specializing to the case $k=1$ gives
\begin{equation*}
  \label{eq:k1-1-ideal}
  \mathcal{I}_{p,1,0}= \langle
    \lambda_i\lambda_j \mid 1\le  i<j\le p \rangle
\end{equation*}
which is a monomial ideal.

\begin{lemma}\label{lemma:k=1,r=0}
    Let $\Sigma_0\in \R^{p\times p}$ be a diagonal matrix with positive diagonal entries. Then
    \[
    \min_{\Lambda} \rlct_\Lambda(\mathcal{I}_{p,1,0};1) = \rlct_0(\mathcal{I}_{p,1,0};1) = (p/2,1),
    \]
    where the minimum ranges over all $\Lambda \in \R^{p\times 1}$ with $\diag(\psi) + \Lambda\Lambda^T = \Sigma_0$ for some $\psi \in  \R_{ > 0}^p$.
\end{lemma}

\begin{proof}
    The first equality follows by Fact~\ref{fact:homogeneous}. For the second, we consider the Newton polyhedron method that we recalled in Section~\ref{subsection:newton-polyhedra}. We write $\mathcal{I} = \mathcal{I}_{p,1,0}$. The Newton polyhedron of $\mathcal{I}$ is the convex positive real cone
    \[
    \mathcal{P}(\mathcal{I}) = \conv \left\{
     e_i+e_j+\xi \ \bigg\vert \ 1\le i<j\le p,\: \xi\in \R^p_{\geq 0}
    \right\}
    \]
    generated by the exponent vectors $e_i + e_j$ of the monomial generators of $\mathcal{I}$, where $e_i$ denotes the $i$-th unit vector. Clearly, all points in $x \in \mathcal{P}(\mathcal{I})$ satisfy the inequality
\[
\sum_{i=1}^m x_i \ge 2.
\]
The set of points where equality holds, that is, the convex hull of the vectors $e_i + e_j$, forms a facet of $\mathcal{P(I)}$.
This facet contains the point $(2/p, \ldots, 2/p)$ in its relative
interior.  Hence, for $\tau=0$ which is the exponent vector of the face function $1$ considered as a monomial, the $\tau$-distance of
$\mathcal{P(I)}$ is $2/p$ and the $\tau$-multiplicity is one. The statement of the lemma now follows from Fact~\ref{fact:monomial}.
\end{proof}

\begin{cor}
  \label{prop:k=1,r=0}
  The learning coefficient and its order of the factor analysis model with $p$-dimensional observations and $k = 1$ latent factor along the submodel with $r = 0$ latent factors satisfy
  \[
  \ell_{10} = \ell_1(\Sigma_0) = \frac{3p}{4} \quad\text{and}\quad m_{10} = m_1(\Sigma_0) = 1,
  \]
  where $\Sigma_0$ is any diagonal matrix with positive real diagonal entries.
\end{cor}

\begin{proof}
This follows directly from Lemma~\ref{lemma:k=1,r=0} and Theorem~\ref{fact:reduction}.
\end{proof}

\begin{lemma}\label{lemma:r=0}
     If $p\ge k+2$, then
  \[
    \min_\Lambda \rlct_\Lambda(\mathcal{I}_{p,k,0}; 1) = \rlct_0(\mathcal{I}_{p,k,0}; 1) = ({pk}/{2},1),
  \]
  where the minimum ranges over all $\Lambda \in \R^{p\times k}$ with $\diag(\psi) + \Lambda\Lambda^T = \Sigma_0$ for some $\psi \in  \R_{ > 0}^p$.
\end{lemma}

\begin{proof}
    Again, the first equality follows by Fact~\ref{fact:homogeneous}. For the second, we proceed by induction on $k$.  If $k=0$, then the claim follows from Lemma~\ref{lemma:model-dimension} and, although technically unnecessary, the claim also follows for $k=1$ by Lemma~\ref{lemma:k=1,r=0}.
  
  Now, for the induction step, fix $k > 0$.  The first transformation we use is a blow-up at the origin, see Section~\ref{subsection:blow-up}.
  By symmetry, we only need to consider one chart and we can write the transformation as
  \begin{align*}
    \lambda_{pk} &= \lambda_{pk}' \text{ and } \\
    \lambda_{ij} &= \lambda_{pk}'\lambda_{ij}' \text{ for }  (i,j)\not=(p,k),
  \end{align*}
  which has Jacobian determinant equal to $(\lambda_{pk}')^{pk-1}$.  Omitting the primes on
  the new variables and letting
  $\tilde\lambda_i=(\lambda_{i1},\dots,\lambda_{i(k-1)})$, the pullback of the ideal
  $\mathcal{I}_{p,k,0}$ is
  \[
    \langle\lambda_{pk}^2\rangle\cdot (\mathcal{J}_1+\mathcal{J}_2),
  \]
  where
  \begin{align*}
    \mathcal{J}_1 &= \langle \lambda_{jk} +\tilde\lambda_p\tilde\lambda_j^T
    \mid
   1\leq j \leq p-1\rangle, \\
    \mathcal{J}_2 &= \langle \lambda_{ik}\lambda_{jk}
    +\tilde\lambda_i\tilde\lambda_j^T \mid 1\le i<j\le p-1\rangle.
  \end{align*}
  It follows from the very definition of the RLCT that
  \[
    \rlct_0(\langle\lambda_{pk}^2\rangle; \lambda_{pk}^{pk-1}) = \frac{pk}{2},
  \]
  see also~\cite[Theorem 7.1]{lin:2014}.
  Using the form of the generators for $\mathcal{J}_1$, we infer
  \begin{align*}
    \mathcal{J}_1+\mathcal{J}_2&=\mathcal{J}_1+\langle\;
    (-\tilde\lambda_p\tilde\lambda_i^T)(-\tilde\lambda_p\tilde\lambda_j^T)+\tilde\lambda_i\tilde\lambda_j^T
    \mid 1\le i<j\le p-1\;\rangle\\
    &= \mathcal{J}_1 + \langle\;
    \tilde\lambda_i(\tilde\lambda_p^T\tilde\lambda_p + I)\tilde\lambda_j^T
    \mid 1\le i<j\le p-1\;\rangle,
  \end{align*}
  where $I$ denotes the identity matrix of dimension $p-1$.
  The first equality follows by subtracting the product of the $i$-th and the $j$-th generator of $\mathcal{J}_1$ from the corresponding generator of $\mathcal{J}_2$ while the second is simply matrix multiplication. Indeed, $(-\tilde{\lambda}_p \tilde{\lambda}_i^T)(-\tilde{\lambda}_p \tilde{\lambda}_j^T) = (\tilde{\lambda}_i\tilde{\lambda}_p^T)(\tilde{\lambda}_p \tilde{\lambda}_j^T) = \tilde{\lambda}_i(\tilde{\lambda}_p^T \tilde{\lambda}_p)\tilde{\lambda}_j^T$.
  
  Now, transform the variables $\lambda_{jk}$ for $1\le j\le p-1$ as
  \[
    \lambda_{jk}=\lambda_{jk}'-\tilde\lambda_p\tilde\lambda_j^T,
  \]
  leaving all other variables fixed.  This map has
  Jacobian determinant equal to $1$ and shows that the $\rlct_0$ of $\mathcal{J}_1+\mathcal{J}_2$
  is equal to the $\rlct_0$ of
  \begin{equation}\label{equ:step-ideal}
    \langle \lambda_{jk} \::\:     1\le j\le p-1\rangle + \langle\;
    \tilde\lambda_i(I+\tilde\lambda_p^T\tilde\lambda_p)\tilde\lambda_j^T
    \::\: 1\le i<j\le p-1\;\rangle.
  \end{equation}
  By Example~\ref{example:linear-space}
  \begin{equation*}
    \rlct_0(\langle \lambda_{jk} \::\:     1\le j\le p-1\rangle;1) = (p-1,1).
  \end{equation*}
  Since the matrix $I+\tilde\lambda_p^T\tilde\lambda_p$ is positive definite and thus has all of its eigenvalues real, positive and bounded away from $0$, we may change coordinates as
  \begin{equation*}
    \tilde\lambda_i =
    \tilde\lambda_i'(I+\tilde\lambda_p^T\tilde\lambda_p)^{-1/2}, \quad
    1\le i\le p-1,
  \end{equation*}
  where the positive definite matrix square root is given by
  \[
 I + \left( \frac{1}{\sqrt{1 + \tilde{\lambda}_p \tilde{\lambda}_p^T}} - 1 \right) \cdot \frac{\tilde{\lambda}_p^T \tilde{\lambda}_p}{\tilde{\lambda}_p \tilde{\lambda}_p^T}.
\]
Hence, the Jacobian determinant of this transformation is positive on its domain and can be ignored. 
  Dropping the primes, this gives 
  \begin{align*}
    &\rlct_0(
    \langle\;
    \tilde\lambda_i(I+\tilde\lambda_p\tilde\lambda_p^T)\tilde\lambda_j^T
    \::\: 1\le i<j\le p-1\;\rangle; 1) \\
    =&\rlct_0(
    \langle\;
    \tilde\lambda_i\tilde\lambda_j^T
    \::\: 1\le i<j\le p-1\;\rangle; 1)=\rlct_0(I_{p-1,k-1,0};1).
  \end{align*}
 Using the induction hypothesis and the sum rule (Fact~\ref{fact:sum-product-rule}(1)), we obtain that the $\rlct_0$ of the ideal in (\ref{equ:step-ideal}) is
  \begin{equation}\label{equation:ind-step}
       \left(p-1+\frac{(p-1)(k-1)}{2},1\right) = \left(\frac{(p-1)(k+1)}{2},1\right).
  \end{equation}
  If $p\geq k+2$, as we assume, then
  \begin{equation*}
    \frac{(p-1)(k+1)}{2}> \frac{pk}{2}.
  \end{equation*}
 By the product rule (Fact~\ref{fact:sum-product-rule}(2)), we conclude that
  \begin{equation*}
    \rlct_0(\mathcal I_{p,k,0}) = 
    \rlct_0(\langle\lambda_{pk}^2\rangle;
    \lambda_{pk}^{pk-1})=\left(\frac{pk}{2},1\right). 
    \qedhere
  \end{equation*}
\end{proof}

\begin{cor}\label{thm:r=0}
The learning coefficient and its order of the factor analysis model with $p$-dimensional observations and $k\leq p-2$ latent factors along the submodel with $r = 0$ latent factors satisfy
  \[
  \ell_{k0} = \ell_k(\Sigma_0)= 
      \frac{p(k+2)}{4}
\quad \text{and} \quad 
  m_{k0} = m_k(\Sigma_0) = 1,
  \]
  where $\Sigma_0$ is any diagonal matrix.
\end{cor}

\begin{proof}
    This follows immediately from Lemma~\ref{lemma:r=0} and Theorem~\ref{fact:reduction}.
\end{proof}

\subsection{A general upper bound}\label{subsection:upper-bound}

In the following lemma, we denote the first entry of the RLCT by $ \operatorname{rlct}_\Lambda(\mathcal{I}_{p,k}(\Sigma_0); 1)$. Note that the cases of Lemma~\ref{lemma:bound} and Corollary~\ref{thm:bound-alternate} with $d_r > p(p+1)/2$ are covered by Lemma~\ref{lemma:model-dimension}.

\begin{lemma}
  \label{lemma:bound}
  Let $\Sigma_0 \in \R^{p\times p}$ be a fixed covariance matrix that is chosen generically from the $r$-factor model for some $r\in\{0,\dots,k\}$.
   If $d_r \leq p(p+1)/2$, then
  \begin{equation*}
   \min_\Lambda \operatorname{rlct}_\Lambda(\mathcal{I}_{p,k}(\Sigma_0); 1) \le 
       \frac{ pk+r(p-k+1)}{2},
  \end{equation*}
  where the minimum ranges over all $\Lambda \in \R^{p \times k}$ with $\Sigma_k(\psi,\Lambda) = \Sigma_0$ for some $\psi \in  \R_{ > 0}^p$.
\end{lemma}

\begin{proof}
  By the considerations in Section~\ref{subsection:LQ-decomp}, we may assume that $\Sigma_0 = \diag(\psi) + LL^T$, where $\psi \in  \R_{ > 0}^p$ and $L \in \R^{p\times r}$ of the form
  \[
  L = 
    \begin{pmatrix}
      L_{11} \\
      L_{21}
    \end{pmatrix}
  \]
  with $L_{11}$ a lower triangular $(r\times r)$-matrix of full rank. Moreover, by Fact~\ref{fact:LQ}, we can replace in the definition of $\mathcal{I}_{p,k}(\Sigma_0)$ the $(p\times k)$-matrix of indeterminates $\Lambda$ by a $(p \times k)$-matrix of indeterminates of the form
  \[
  \Lambda =
  \begin{pmatrix}
      \Lambda_{11} & 0\\
      \Lambda_{21} & \Lambda_{22}
    \end{pmatrix},
  \]
  where $\Lambda_{11}$ is a lower triangular $(r\times r)$-matrix on whose diagonal entries we impose that they be positive. We call this representation \textit{LQ-coordinates}, but note that lower triangularity is only imposed on the $\Lambda_{11}$ block and need not be imposed on the $\Lambda_{22}$ block. 
  In these coordinates, the generators of $\mathcal{I}_{m,k}(\Sigma_0)$ are the off-diagonal
  entries of the matrix
  \begin{multline}
    \label{eq:split}
    \begin{pmatrix}
      \Lambda_{11} & 0\\
      \Lambda_{21} & \Lambda_{22}
    \end{pmatrix}
    \begin{pmatrix}
      \Lambda_{11} & 0\\
      \Lambda_{21} & \Lambda_{22}
    \end{pmatrix}^T
    -\begin{pmatrix}
      L_{11}\\
      L_{21}
    \end{pmatrix}
    \begin{pmatrix}
      L_{11}\\
      L_{21}
    \end{pmatrix}^T\\
    =
    \left[\begin{pmatrix}
      \Lambda_{11}\\
      \Lambda_{21}
    \end{pmatrix}
    \begin{pmatrix}
      \Lambda_{11}\\
      \Lambda_{21}
    \end{pmatrix}^T
    -\begin{pmatrix}
      L_{11}\\
      L_{21}
    \end{pmatrix}
    \begin{pmatrix}
      L_{11}\\
      L_{21}
    \end{pmatrix}^T
    \right]+
    \begin{pmatrix}
      0 & 0\\
      0 & \Lambda_{22}\Lambda_{22}^T
    \end{pmatrix}.
  \end{multline}
  In particular, each generator is the sum of a term formed from the entries
  of $(\Lambda_{11},\Lambda_{21})$ and a second term formed from the
  entries of $\Lambda_{22}$.

  We have that the off-diagonal entries
  of the first matrix on the right-hand side of~(\ref{eq:split}) in
  square brackets generate the ideal $\mathcal{I}_{p,r}(\Sigma_0)$ in LQ-coordinates, and the off-diagonal entries of the second matrix
  generate $\mathcal{I}_{p-r,k-r}(0)$.  Applying the inequality
  \begin{equation*}
    \label{eq:split-square}
    (x+y)^2\le 2x^2+2y^2
  \end{equation*}
  to each generator in turn, we see that the sum of squares of the generators of $\mathcal{I}_{p,k}(\Sigma_0)$ is bounded above by two times the sum of
  the two sums of squares that are associated with
  $\mathcal{I}_{p,r}(\Sigma_0)$ and $\mathcal{I}_{p-r,k-r}(0)$ respectively.
  Using Fact~\ref{fact:homogeneous}, we deduce that
  \begin{equation*}
    \operatorname{rlct}_\Lambda(\mathcal{I}_{p,k}(\Sigma_0); 1) \le     \operatorname{rlct}_{\Lambda_{11}}(\mathcal{I}_{p,r}(\Sigma_0);
    1) +\operatorname{rlct}_0(\mathcal{I}_{p-r,k-r}(0);1)
  \end{equation*}
  for every $\Lambda \in \R^{p\times k}$ in LQ-coordinates.
  The claimed upper bound now follows, because
  \begin{equation*}
    \operatorname{rlct}_{\Lambda_{11}}(\mathcal{I}_{p,r}(\Sigma_0);1) = 
       pr-\frac{r(r-1)}{2}
  \end{equation*}
  by Lemma~\ref{lemma:model-dimension}
  and
  \begin{equation*}
    \operatorname{rlct}_0(\mathcal{I}_{p-r,k-r}(0);1) \leq \frac{(p-r)(k-r)}{2}
  \end{equation*}
  by Lemma~\ref{lemma:r=0}.
\end{proof}


\begin{cor}
\label{thm:bound-alternate}
     Let $\ell_k(\Sigma_0)$ be the learning coefficient of the factor analysis model with $k$ latent factors at a fixed generic covariance matrix $\Sigma_0$ in the $r$-factor model, where $r \in \{0,\ldots,k\}$. If $d_r = p(r+1) - r(r-1)/2 \leq p(p+1)/2$, then 
  \[
  \ell_k(\Sigma_0) \leq \frac{p(k+2)+r(p-k+1)}{4}.
  \]
\end{cor}

Note that we already can see at this point that the inequality just given is an equality for $r=0$ (Corollary~\ref{thm:r=0}) and
$r=k$ (Lemma~\ref{lemma:model-dimension}). The values for $1\le r\le k-1$ are equally spaced between
those extremes. The case $d_r > p(p+1)/2$ is covered by Lemma~\ref{lemma:model-dimension}. The main result of this paper is that this bound gives the correct value for the learning coefficient, see Section~\ref{subsection:formula-learning-coefficients}.

\begin{proof}
    This follows immediately from Lemma~\ref{lemma:bound} and Theorem~\ref{fact:reduction}.
\end{proof}

\begin{remark} \label{rmk:nicer-eqs}
We continue the line of thought from the proof of Corollary \ref{thm:bound-alternate} by applying one further coordinate transformation and manipulations on the generators of the resulting ideals. This will be used in Section~\ref{subsection:one-factor} below to handle the case $r = 1$. Note that
\begin{multline}
  \label{eq:gens-in-matrix}
  \left[\begin{pmatrix}
      \Lambda_{11}\\
      \Lambda_{21}
    \end{pmatrix}
    \begin{pmatrix}
      \Lambda_{11}\\
      \Lambda_{21}
    \end{pmatrix}^T
    -\begin{pmatrix}
      L_{11}\\
      L_{21}
    \end{pmatrix}
    \begin{pmatrix}
      L_{11}\\
      L_{21}
    \end{pmatrix}^T
  \right]+
  \begin{pmatrix}
    0 & 0\\
    0 & \Lambda_{22}\Lambda_{22}^T
  \end{pmatrix}\\
  =
  \begin{pmatrix}
    \Lambda_{11}\Lambda_{11}^T-L_{11}L_{11}^T &
    \Lambda_{11}\Lambda_{21}^T-L_{11}L_{21}^T \\ 
    \Lambda_{21}\Lambda_{11}^T-L_{21}L_{11}^T &
    \Lambda_{22}\Lambda_{22}^T+\Lambda_{21}\Lambda_{21}^T-L_{21}L_{21}^T
  \end{pmatrix}.
\end{multline}
Since $\Lambda_{11}$ is invertible in LQ-coordinates, we
may apply the real analytic isomorphism using
\begin{equation*}
  \Lambda_{21}' = \Lambda_{21}\Lambda_{11}^T-L_{21}L_{11}^T
\end{equation*}
leaving the variables in $\Lambda_{11}$ and $\Lambda_{22}$ fixed, whose inverse on the $\Lambda_{21}$-coordinates is
\begin{equation*}
    \Lambda_{21} =
    \left(\Lambda_{21}'+L_{21}L_{11}^T\right)\Lambda_{11}^{-T}.
\end{equation*}
In the new coordinates, after dropping primes, our ideal is 
generated by the off-diagonal entries of
\begin{equation*}
\begin{footnotesize}
  \begin{pmatrix}
    \Lambda_{11}\Lambda_{11}^T-L_{11}L_{11}^T &
    \Lambda_{21}^T \\ 
    \Lambda_{21} &
    \Lambda_{22}\Lambda_{22}^T+\left(\Lambda_{21}+L_{21}L_{11}^T\right)\Lambda_{11}^{-T}\Lambda_{11}^{-1}\left(\Lambda_{21}^T+L_{11}L_{21}^T\right)-L_{21}L_{21}^T
  \end{pmatrix}.
  \end{footnotesize}
\end{equation*}
We may use the generators in $\Lambda_{21}$ to remove terms from the
generators in the lower right block of the matrix. Doing so, we obtain that the
ideal is generated by the off-diagonal entries of
\begin{equation}
  \begin{pmatrix}
    \Lambda_{11}\Lambda_{11}^T-L_{11}L_{11}^T &
    \Lambda_{21}^T \\ 
    \Lambda_{21} &
    \Lambda_{22}\Lambda_{22}^T+L_{21}L_{11}^T\Lambda_{11}^{-T}\Lambda_{11}^{-1}L_{11}L_{21}^T-L_{21}L_{21}^T
  \end{pmatrix}.
\end{equation}
Since $\Lambda_{21}$ has $(p-r)r$ entries, it follows by Fact~\ref{fact:sum-product-rule}(1) that the RLCT of
the ideal is $((p-r)r,0)$ plus the RLCT of the ideal generated by the
off-diagonal entries of the two matrices
\begin{equation}
\label{eq:matrix1}
  \Lambda_{11}\Lambda_{11}^T-L_{11}L_{11}^T
\end{equation}
and
\begin{equation}
  \label{eq:matrix2}
  \Lambda_{22}\Lambda_{22}^T+L_{21}L_{11}^T\left[ \left( \Lambda_{11}\Lambda_{11}^T \right)^{-1}-\left(L_{11}L_{11}^T\right)^{-1}\right]L_{11}L_{21}^T.
\end{equation}
The former matrix is of size $r\times r$ and the latter is of size
$(p-r)\times (p-r)$.
\end{remark}

\subsection{The general formula for the learning coefficient}\label{subsection:formula-learning-coefficients}
We now prove our main result, a formula showing that, for arbitrary $p\geq k \geq r \geq 0$, the upper bound from Corollary~\ref{thm:bound-alternate} is the actual value of the learning coefficient of a factor analysis model and that its order is $1$. Along the way, we compute a full log resolution of singularities for the real analytic space defined by the fiber ideal on the parameter space of the factor analysis model.

\begin{theorem}
\label{thm:main-formula}
     The learning coefficient and its order for the factor analysis model with $k$ latent factors at a fixed generic covariance matrix $\Sigma_0$ in the $r$-factor model with $r \in \{0,\ldots,k\}$ satisfy 
  \[
  \ell_k(\Sigma_0) = \ell_{kr} = \frac{p(k+2)+r(p-k+1)}{4} \quad
  \text{ and } \quad
    m_k(\Sigma_0) = m_{kr} = 1.
  \]

\end{theorem}

\begin{proof}
Let $\Sigma_0 = LL^T + D$, where $L \in \R^{p\times k}$ is of rank $r$ and $D = \diag(d_1,\ldots,d_p)$ with all $d_i$ positive. Suppose that $r$ is minimal such that such a decomposition is possible; this amounts to a first genericity assumption on $\Sigma_0$. We can use an orthonormal $p\times p$ matrix $Q$ to diagonalize $LL^T$, so that
\[
E \coloneqq QLL^TQ^T = \diag(\varepsilon_1,\ldots,\varepsilon_r,0,\ldots,0)
\]
with $\varepsilon_i \neq 0$.
The (full) fiber ideal $\overline{\mathcal{I}}_{p,k}(\Sigma_0)$ at $\Sigma_0$ of the $p$-factor model with $k$ latent factors is generated by the entries of $\Lambda\Lambda^T + \diag(\psi_1,\ldots,\psi_p) - \Sigma_0$. Multiplying this matrix with $Q$ from the left and with $Q^T$ from the right changes the generators of $\overline{\mathcal{I}}_{p,k}(\Sigma_0)$ but not the ideal, as $Q$ is invertible. Changing variables via $\Lambda \mapsto Q^T\Lambda$, which is a real analytic isomorphism as $Q$ is invertible, we get
\begin{equation}\label{equ:first-Q-form}
   \Lambda\Lambda^T + Q\diag(\psi_1- d_1, \ldots,\psi_p-d_p)Q^T - E. 
\end{equation}
Note that the Jacobian determinant is non-vanishing by the very fact that this is an isomorphism.
We apply two further changes of coordinates. The first one, $\psi_i \mapsto \psi_i + d_i$, is again an isomorphism, while the second one, $\psi_i \mapsto \psi_i^2$, is an isomorphism away from the set of $\psi$ with $\psi_1 \cdots \psi_p = 0$. It gives a Jacobian determinant of $2^p \psi_1\cdots \psi_p$ which we have to take into account when applying the chain rule (Fact~\ref{fact:chain-rule}). We write $A(\psi) = Q\diag(\psi_1^2, \ldots,\psi_p^2)Q^T$. Note that $A(\psi)_{ij} = \sum_{s = 1}^p q_{is}q_{js}\psi_s^2$ only depends on the $i$-th and $j$-th row of $Q = (q_{is})$. Let us assume that $\Lambda$ is in $LQ$-coordinates, see Section~\ref{subsection:LQ-decomp}. Then, taking into account the transformations, equation (\ref{equ:first-Q-form}) simplifies to
\[
    \begin{bmatrix}
        \Lambda_{11}\Lambda_{11}^T + A(\psi)_{11} - \diag(\varepsilon_1,\ldots,\varepsilon_r) & \Lambda_{11}\Lambda_{21}^T + A(\psi)_{21}^T\\
        \Lambda_{21}\Lambda_{11}^T + A(\psi)_{21} & \Lambda_{21}\Lambda_{21}^T + \Lambda_{22}\Lambda_{22}^T + A(\psi)_{22}
    \end{bmatrix},
\]
where $A(\psi)_{11}$ is the $r\times r$ upper left submatrix of $A(\psi)$, $A(\psi)_{22}$ is the $(p-r) \times (p-r)$ lower right submatrix of $A(\psi)$, and $A(\psi)_{21}$ is the $(p-r)\times r$ lower left submatrix of $A(\psi)$.
In $LQ$-coordinates the submatrix $\Lambda_{11}$ is an invertible lower triangular matrix with positive diagonal entries. So, the map $\Lambda_{21}\mapsto \Lambda_{21}\Lambda_{11}^T$ followed by $\Lambda_{21} \mapsto \Lambda_{21} + A(\psi)_{21}$ is a real analytic isomorphism that transforms the matrix of generators to
\begin{footnotesize}
\[ 
    \begin{bmatrix}   
\Lambda_{11}\Lambda_{11}^T + A(\psi)_{11} - \diag(\varepsilon_1,\ldots,\varepsilon_r) & \Lambda_{21}^T \\
        \Lambda_{21} & (\Lambda_{21} - A(\psi)_{21})\Lambda_{11}^{-T}\Lambda_{11}^{-1}(\Lambda_{21} - A(\psi)_{21})^T + \Lambda_{22}\Lambda_{22}^T + A(\psi)_{22}
\end{bmatrix}. 
\]
\end{footnotesize}
Using the entries of $\Lambda_{21}$, we can eliminate the summands involving these entries in the right lower block which again amounts to a change of generators for the ideal. This leads us to 
\[
\begin{bmatrix}
     \Lambda_{11}\Lambda_{11}^T + A(\psi)_{11} - \diag(\varepsilon_1,\ldots,\varepsilon_r) & \Lambda_{21}^T \\
        \Lambda_{21} & A(\psi)_{21}\Lambda_{11}^{-T}\Lambda_{11}^{-1} A(\psi)_{21}^T + \Lambda_{22}\Lambda_{22}^T + A(\psi)_{22}
\end{bmatrix}.
\]
Using the sum rule (Fact~\ref{fact:sum-product-rule}(1)), we disregard $\Lambda_{21}$ by adding $(p-r)r$ to the RLCT in the end. Using Fact~\ref{fact:chain-rule}, we now blow up along the linear subspace defined by the vanishing of the coordinates in $\Lambda_{22}$ and the $\psi$. Call $\mathcal{I}_1$ and $\mathcal{I}_2$ the ideals generated by the entries of $ \Lambda_{11}\Lambda_{11}^T + A(\psi)_{11} - \diag(\varepsilon_1,\ldots,\varepsilon_r)$ and $A(\psi)_{21}\Lambda_{11}^{-T}\Lambda_{11}^{-1} A(\psi)_{21}^T + \Lambda_{22}\Lambda_{22}^T + A(\psi)_{22}$, respectively. By symmetry of the generators, we only need to consider the charts corresponding to $\psi_p$ and $\lambda_{pk}$. We write the entry of $A(\psi)_{21}\Lambda_{11}^{-T}\Lambda_{11}^{-1} A(\psi)_{21}^T$ at position $ij$ as $f_{ij}$, which is a homogeneous polynomial of degree $4$ in the $\psi$-variables and is independent of the $\Lambda_{22}$ variables. Moreover, the diagonal entries of $\Lambda_{11}^{-T}\Lambda_{11}^{-1}$ are rational functions that are sums of products of squares, and hence, the same is true for the $f_{ii}$.

We start with the chart corresponding to $\psi_p$. Taking into account the Jacobian term from before, the Jacobian determinant is $2^p\psi_1\cdots \psi_{p-1}\cdot \psi_p^{(p-r)(k-r) + 2p -1}$. The diagonal generators of $\mathcal{I}_2$ are of the form
\[
f_{ii} + \sum_{l = r+1}^k \lambda_{il}^2 + \sum_{s = 1}^p q_{is}^2 \psi_s^2
\]
for $i \in \{r+1,\ldots,p\}$.
So their pullback is
\[
\psi_p^2 \cdot \left(g_{ii} +  \sum_{l = r+1}^k \lambda_{il}^2 + \sum_{s = 1}^{p-1} q_{is}^2 \psi_s^2 + q_{ip}^2 \right),
\]
where $\psi_p^2g_{ii}$ is the pullback of $f_{ii}$. Note that $g_{ii}$ is still a sum of squares, because $f_{ii}$ was of degree $4$ in the $\psi$-variables.
If we choose $\Sigma_0$, and hence $L$ generically, we can assume that not all of the $q_{ip}$ for $i \in \{r+1, \ldots, p\}$ are equal to $0$ because $Q$ is used to diagonalize $LL^T$. So, for some $i$ the expression $\left(g_{ii} + \sum_{l = r+1}^k \lambda_{il}^2 + \sum_{s = 1}^{p-1} q_{is}^2 \psi_s^2 + q_{ip}^2 \right)$ is strictly positive over the real numbers, and hence, is an analytic unit. Therefore, the pullback of $\mathcal{I}_2$ is $\langle \psi_p^2\rangle$.

We use this to eliminate the summands of the generators in the pullback of $\mathcal{I}_1$ having $\psi$-terms, which makes this ideal free of $\psi$-coordinates. This leaves us with the entries of $\Lambda_{11}\Lambda_{11}^T - \diag(\varepsilon_1,\ldots,\varepsilon_r)$ which, after applying the analytic isomorphism that scales the $l$-th row of $\Lambda_{11}$ by $\sqrt{\varepsilon_r}>0$ and multiplying the resulting matrix by $\diag(1/\sqrt{\varepsilon_1},\ldots,1/\sqrt{\varepsilon_r})$ from both sides, which amounts to a change of generators but keeps the ideal unchanged, becomes $\Lambda_{11}\Lambda_{11}^T - I_r$. This defines a real analytic manifold of codimension $r(r+1)/2$, and hence, its RLCT is 
\[
\left( \frac{r(r+1)}{2}, 1 \right).
\]
We compute the RLCT of the pullback of $\mathcal{I}_2$ as
\[
\rlct\left(\psi_p^2; \psi_p^{(p-r)(k-r) + 2p -1}\right) = \left( \frac{(p-r)(k-r) + 2p}{2}, 1 \right).
\]
Taking into account the coordinates from $\Lambda_{21}$, the first entry of the RLCT of the fiber ideal on this chart is
\begin{align*}
    & \frac{2(p-r)r + (p-r)(k-r) + 2p + r(r+1)}{2} = \frac{p(k+2) + r(p-k+1)}{2}
\end{align*}
with order $1$. By Theorem~\ref{fact:reduction}, these are the conjectured numbers.

We now do a similar argument on the chart corresponding to $\lambda_{pk}$. The pullback of $\mathcal{I}_2$ contains the generators
\[ \lambda_{pk}^2 \cdot \left(h_{ii} + 1 + \sum_{l = r+1}^{k-1} \lambda_{pl}^2 + \sum_{s = 1}^p q_{ps}^2 \psi_s^2 \right),\]
where $\lambda_{pk}^2h_{ii}$ is again the pullback of $f_{ii}$ and
in which the second factor is an analytic unit by the same argument as before. So, this pullback is $\langle \lambda_{pk}^2 \rangle$ and we have a Jacobian determinant of $2^p\psi_1 \cdots \psi_p \lambda_{pk}^{(p-r)(k-r) +2p - 1}$. The summands containing $\psi$-coordinates in the generators of $\mathcal{I}_1$ are eliminated using $\lambda_{pk}^2$. So, again, we can compute the RLCT of the pullback of $\mathcal{I}_2$ as
\[
\rlct\left(\lambda_{pk}^2; \lambda_{pk}^{(p-r)(k-r) + 2p -1}\right) = \left( \frac{(p-r)(k-r) + 2p}{2}, 1 \right).
\]
The remainder of the argument is exactly the same as for the $\psi_p$-chart.
\end{proof}

\section{Singularities of distinguished factor analysis models}
\label{sec:distinct-singularities}

In this final section, we show that the genericity assumption on $\Sigma_0$ is necessary in general. More precisely, the model $\mathcal M_1$ with $k = 1$ decomposes into exactly three strata with distinct learning coefficients. Furthermore, we analyze the nature of genericity for one-factor covariance matrices by proving that the RLCT takes the expected value for $\Sigma_0 \in \mathcal{M}_1$ with all off-diagonal entries being non-zero and arbitrary $1 \leq k \leq p$.

\subsection{Distinct singularities of one-factor models}\label{subsection:distinct-one-factor-singularities}

We now consider models with $k=1$ factors, in which case
$\Lambda=(\lambda_1, \ldots, \lambda_p)^T$ varies over vectors in
$\mathbb{R}^p$.  
The singularities of the one-factor parametrization
$\Sigma_1(\psi, \Lambda)$ are vectors $\Lambda_0$ with no more than two non-zero entries, see \cite{anderson1956}. 
We distinguish three cases for $\Lambda_0$ leading to distinct values of the learning coefficients at the covariance matrices associated to $\Lambda_0$ under $\Sigma_1$:
\begin{enumerate}
\item[(a)] $\Lambda_0$ has more than two non-zero entries;
\item[(b)] $\Lambda_0$ has exactly two non-zero entries;
\item[(c)] $\Lambda_0$ has at most one non-zero entry.
\end{enumerate}
Case (a) is the generic case (along the $1$-factor model) covered by Lemma~\ref{lemma:model-dimension} where the learning coefficient equals $p$.
In case (c), a covariance matrix $\Sigma_0$ associated with
$\Lambda_0$ is diagonal, which corresponds to the rank $0$ case we discussed in Corollary~\ref{prop:k=1,r=0}. The corresponding learning coefficient is $3p/4$.
In case (b), $\Sigma_0$ has precisely two non-zero off-diagonal entries. This is a singularity
that is not covered by the rank $0$ case and, as we will show, has different asymptotic
behavior:

\begin{prop}
  \label{prop:k1-two-nonzero}
  If $\Sigma_0$ is a positive definite matrix with precisely two
  non-zero off-diagonal entries then 
  \[
  \ell_1(\Sigma_0) = \frac{2p-1}{2} \quad\text{and}\quad
  m_1(\Sigma_0) = 1.
  \]
\end{prop}

\begin{proof}
Suppose without loss of generality that $\sigma_{12} = \sigma_{21}$ are
the two non-zero off-diagonal entries of $\Sigma_0$.  The partial fiber ideal at $\Sigma_0$
is
\begin{equation*}
   \mathcal{I} \coloneqq \mathcal{I}_{p,1}(\Sigma_0)=\left\langle
    \lambda_1\lambda_2-\sigma_{12},\:
    \lambda_1\lambda_i,\: \lambda_j\lambda_l \mid 3 \le i \le p,  2\le j,l\le
    p,\, j\not=l\right\rangle.
\end{equation*}
The ideal contains
\begin{equation*}
  \label{eq:k1-two-change-generator}
    \lambda_2\cdot(\lambda_1\lambda_i) 
     - \lambda_i\cdot(\lambda_1\lambda_2-\sigma_{12})=
   \sigma_{12}\cdot \lambda_i
\end{equation*}
for all $i\ge 3$.  It follows that
\begin{equation*}
  \mathcal{I}=\left\langle
    \lambda_1\lambda_2-\sigma_{12},\:
    \lambda_i \mid 3\le  i\le
    p\right\rangle.
\end{equation*}

In light of Fact~\ref{fact:chain-rule}, we apply the map
\begin{equation*}
  \label{eq:hyperbola-change-coords}
  \lambda_1 = (\lambda_1'+\sigma_{12})/\lambda_2', \quad \lambda_i = \lambda_i' \text{ for } i \geq 2.
\end{equation*}
Dropping the primes, this map has Jacobian determinant $1/|\lambda_2|$. As $\sigma_{12} \neq 0$, the ideal $\mathcal{I}$ does not vanish at any point of $ U = \{\lambda_2 = 0\}$ and hence we can compute the $\rlct$ away from $U$. The Jacobian determinant can therefore be ignored.  The pullback of the 
ideal $\mathcal{I}$ under this map is
\begin{equation*}
  \left\langle
    \lambda_1,\:
    \lambda_i \mid 3\le  i\le
    p\right\rangle.
\end{equation*}
By Example~\ref{example:linear-space}, we infer that $\min_\Lambda \rlct_\Lambda(\mathcal{I};1) = (p-1,1)$, where the minimum ranges over all $\Lambda \in \R^{p\times 1}$ with $\Sigma_0 = \diag(\psi) + \Lambda\Lambda^T$ for some $\psi \in  \R_{ > 0}^p$. As usual, Theorem~\ref{fact:reduction} yields the learning coefficient as conjectured.
\end{proof}

\subsection{Genericity for one-factor covariance matrices}\label{subsection:one-factor} In the previous section, we studied the distinct singularities of the model $\mathcal{M}_1$ with $k = 1$ factors. We now turn to the case where $k$ is arbitrary. We determine the exact polynomial relations defining a Zariski open subset $U_{k1}$ of covariance matrices $\Sigma_0$ in the one-factor model $\mathcal{M}_1$ such that the learning coefficients $\ell_k(\Sigma_0)$ and their orders $m_k(\Sigma_0)$ do not depend on the choice of $\Sigma_0 \in U_{k1}$. Indeed, $U_{k1}$ is determined as the set of all $\Sigma_0 \in \mathcal{M}_1$ all of whose off-diagonal entries are non-zero. Recall the relevance of such a genericity assumption in regard to the sBIC, see Section~\ref{subsection:sbic}.

\begin{lemma}\label{lemma:k=2,r=1}
    Let $p> 1$, and let $\Sigma_0 \in \R^{p\times p}$ in the $1$-factor model with $p$-dimensional observations be chosen generically (in the sense that all of its off-diagonal entries are non-zero). If $p > 3$ then
    \[
    \min_{\Lambda} \rlct_\Lambda(\mathcal{I}_{p,2}(\Sigma_0);1) = \left(\frac{3p-1}{2},1\right)
    \]
    and
    \[
    \min_{\Lambda} \rlct_\Lambda(\mathcal{I}_{2,2}(\Sigma_0);1) = \left(1,1\right),
    \]
    \[
    \min_{\Lambda} \rlct_\Lambda(\mathcal{I}_{3,2}(\Sigma_0);1) = \left(3,1\right),
    \]
    where the minima range over all $\Lambda \in \R^{p\times 2}$ with $\diag(\psi) + \Lambda\Lambda^T = \Sigma_0$ for some $\psi \in  \R_{ > 0}^p$ and analogously for $p \in\{2,3\}$, respectively.
\end{lemma}


\begin{proof}
The case $p = 2$ is covered by Lemma~\ref{lemma:model-dimension}(2). So, suppose $p\geq3$.
The covariance matrix $\Sigma_0$ being in $\mathcal{M}_1$ means that its off-diagonal entries are products $\gamma_{i}\gamma_{j}$ of the entries of a
vector $\gamma\in\mathbb{R}^p$.  The genericity assumption in the statement of the lemma translates to all
entries of $\gamma$ being non-zero.  By Lemma~\ref{lem:torus}, we may
assume that every entry of the vector $\gamma$ is equal to $1$. Consequently, all of the off-diagonal entries of
$\Sigma_0$ are equal to 1.

By Fact~\ref{fact:LQ}, we may assume that the factor loading matrix is of the form
\[
\Lambda=\begin{pmatrix}
  \lambda_{11} &  0\\
  \lambda_{21} & \lambda_{22} \\
  \vdots & \vdots
\end{pmatrix}.
\]
The partial fiber ideal is then
\[
\mathcal{I} = \mathcal{I}_{p,2}(\Sigma_0) =
\left\langle  
  \lambda_{11}\lambda_{i1}-1,\,  \lambda_{i1}\lambda_{j1}+\lambda_{i2}\lambda_{j2}-1 \mid 2\le i, j\le p,\, i\not= j
\right\rangle.
\]
Since
$\lambda_{j1}\cdot (\lambda_{11}\lambda_{i1}-1) - \lambda_{i1}\cdot (\lambda_{11}\lambda_{j1}-1) = \lambda_{i1}-\lambda_{j1}$
we obtain
\begin{equation*}
  \mathcal{I} =
  \left\langle
    \lambda_{11}\lambda_{p1}-1,\, \lambda_{i1}-\lambda_{p1},\, \lambda_{p1}^2+\lambda_{i2}\lambda_{j2}-1 \::\: 2\le i, j\le p,\, i\not= j
  \right\rangle.
\end{equation*}
We apply the real analytic isomorphism
\begin{align*}
  \lambda_{11} &= \lambda_{11}', & \lambda_{i1}&=\lambda_{i1}'+\lambda_{p1}' \quad(2\le i\le p-1), & \lambda_{p1}=\lambda_{p1}'
\end{align*}
whose Jacobian determinant equals $1$.
After dropping primes, the pullback of the ideal $\mathcal{I}$  is
\begin{equation}
  \label{eq:separate}
  \mathcal{J}+\langle \lambda_{i1} \mid 2\le i\le p-1\,\rangle,
\end{equation}
with
\begin{equation}
  \label{eq:nicer-generators}
  \mathcal{J} =
  \langle 
    \lambda_{11}\lambda_{p1}-1,\, \lambda_{p1}^2+\lambda_{i2}\lambda_{j2}-1 \mid 2\le i, j\le p, i \neq j
  \rangle.
\end{equation}
In the sequel, we compute the RLCT associated with
$\mathcal{J}$ whose second component (the multiplicity) will be computed as $1$.  Using the sum rule (Fact~\ref{fact:sum-product-rule}(1)) and Example~\ref{example:linear-space}, the first component of the RLCT for $\mathcal{I}$ is larger than the one of $\mathcal{J}$
by $p-2$ and its second component is $1$.

To find the RLCT of $\mathcal{J}$ we compute a decomposition of the variety $V(\mathcal{J})$ (whose ambient space is a $(p+1)$-dimensional affine space) into smaller varieties. Note that the extension of $\mathcal{J}$ to the ring $\R[\lambda_{11},\lambda_{p1},\lambda_{22},\ldots,\lambda_{p2}][(\lambda_{p1}^2-1)^{-1}]$ equals the ideal
\[
Q = \langle \lambda_{11}\lambda_{p1}-1,\, \lambda_{i2}-\lambda_{p2},\,
  \lambda_{p1}^2+\lambda_{p2}^2-1 \mid 2\le i\le p-1 \rangle.
\]
Its $(p+1) \times p$ Jacobian matrix has the following form:
\[
\left[
\begin{array}{ccccc}
\lambda_{p1} & 0 & \cdots  & 0 & 0  \\
\lambda_{11} & 0 & \cdots & 0 & 2\lambda_{p1}  \\
0 & 1 &  &  & 0  \\
\vdots & & \ddots & & \vdots \\
\vdots & & & 1 & 0 \\
0 & -1 & \cdots  &  -1 & 2\lambda_{p2}
\end{array}
\right],
\]
This matrix has full rank $p$ whenever $\lambda_{11}\lambda_{p1} \neq 0$ which is never the case on the variety defined by $Q$ as $\lambda_{11}\lambda_{p1} -1 \in Q$. So, $Q$ defines a (complex) smooth curve. 

The intersection of $V(\mathcal{J})$ with the two hyperplanes $\lambda_{p1} = 1$ and $\lambda_{p1} = -1$ decompose into the lines defined by the ideals
\begin{align*}
  I_l = \langle  \lambda_{11}-1, \lambda_{p1}-1,\, \lambda_{i2} \mid 2\le i\le p,\, i\not=l\rangle,  \ l=2,\dots,p
  \end{align*}
  and
  \begin{align*}
  \langle  \lambda_{11}+1, \lambda_{p1}+1,\, \lambda_{i2} \mid 2\le i\le p,\, i\not=l\rangle, \ l=2,\dots,p,
\end{align*}
respectively. Consequently, $\mathcal{J}$ is the intersection of these ideals and $Q$, and $V(\mathcal{J})$ is the union of the smooth varieties defined by them.
In particular, the singular locus of
$V(\mathcal{J})$ consists of the pairwise intersections of these components.
Only two points lie in such intersections, namely,
\begin{align*}
  w&=(\lambda_{11},\lambda_{p1}, \lambda_{22},\dots,\lambda_{p2})=(1,1,0,\dots,0)
\end{align*}
and $w'=-w$. Outside these points, the RLCT of $\mathcal{J}$ equals $(p,1)$ by Example~\ref{example:linear-space}.  By symmetry,
$\rlct_{w}(\mathcal{J};1)=\rlct_{w'}(\mathcal{J};1)$, and we will see, a posteriori, that this pair is smaller than $(p,1)$. So, for the remainder of this proof, it suffices to consider the point $w$.

We begin by translating $w$ to the origin by changing $\lambda_{i 1}$-variables as
$\lambda_{11}=\lambda_{11}'+1$ and $\lambda_{p1}=\lambda_{p1}'+1$.  The pullback of $\mathcal{J}$ under this map is
\begin{align*}
  \mathcal{J}_w = \langle \lambda_{11}\lambda_{p1}+\lambda_{11}+\lambda_{p1},\, \lambda_{p1}^2+2\lambda_{p1}+\lambda_{i2}\lambda_{j2}
  \mid 2\le i<j\le p \rangle.
\end{align*} %
We further apply the real analytic map defined by
\begin{align*}
  \lambda_{11}'&=\lambda_{11}\lambda_{p1}+\lambda_{11}+\lambda_{p1}, &
  \lambda_{p1}'&=\lambda_{p1}^2+2\lambda_{p1}+\lambda_{22}\lambda_{32}, &
  \lambda_{i2}'&=\lambda_{i2} \quad \text{ for } 2\le i\le p,
\end{align*}
whose Jacobian determinant equals $2(\lambda_{11}+1)^2$, which is non-zero on the real variety defined by $\mathcal{J}_w$, and can hence be ignored.
Dropping the primes and simplifying, the resulting ideal is
\begin{equation}
\label{eq:k12line5}
 \langle \lambda_{11}, \lambda_{p1}\rangle + \mathcal{J}_1,
 \end{equation}
  where
\begin{align}
  \mathcal{J}_1 &= \langle \lambda_{i2}\lambda_{j2} - \lambda_{22}\lambda_{32} : 2 \leq i < j \leq p \rangle \label{equ:first-form} \\
  &= \langle \lambda_{i2}\lambda_{j2} - \lambda_{k2}\lambda_{l2} : 2 \leq i < j \leq p, 2 \leq k < l \leq p \rangle. \label{equ:second-form}
\end{align}
Again, by Fact~\ref{fact:sum-product-rule}(1) and Example~\ref{example:linear-space}, we further reduced the problem to computing $\rlct_0(\mathcal{J}_1;1)$ by adding $2$ to its first component.

For $p = 3$, the ideal $\mathcal{J}_1$ equals $0$, and hence, we compute the desired minimum of RLCTs as $\rlct_0(\langle \lambda_{11},\lambda_{p1}\rangle;1) + (p-2,0) = (2,1) + (1,0) = (3,1)$.

So, suppose that $p > 3$.
Here, we apply the blow-up at the origin to $\mathcal{J}_1$, see Section~\ref{subsection:blow-up}. The  generators of the ideal $\mathcal{J}_1$ in (\ref{equ:second-form}) are invariant under every permutation of the variables. Therefore, all charts of the blow-up are the same and we may just consider one of them, say the chart corresponding to $\lambda_{22}$.
Using the generators in (\ref{equ:first-form}), the pullback of $\mathcal{J}_1$ in this chart is 
\begin{align*}
  \langle
  \lambda_{22}^2(\lambda_{i2}-\lambda_{32}), 
  \lambda_{22}^2(\lambda_{j2}\lambda_{k2}-\lambda_{32}) \mid 3<i, \ 2< j<k\le p
  \rangle,
\end{align*}
and we have the Jacobian determinant $\lambda_{22}^{p-2}$.
Making the change of variables
\begin{align*}
  \lambda_{22}'&=\lambda_{22}, &
  \lambda_{i2}'&=\lambda_{i2}-\lambda_{32} \quad \text{ for } 3<j \le p,
\end{align*}
yields the ideal
\begin{equation}
\label{eq: k12line7}
  \langle
  \lambda_{22}^2\lambda_{32}(\lambda_{32}-1), 
  \lambda_{22}^2\lambda_{i2} \mid 3< i \le p
  \rangle.
  \end{equation}
There are two intersection points between the strict transform defined by $\langle
  \lambda_{32}(\lambda_{32}-1), 
  \lambda_{i2} \mid 3< i \le p
  \rangle$ and the exceptional divisor defined by $\langle \lambda_{22}^2 \rangle$, namely
\begin{align*}
(\lambda_{22},\lambda_{32},\lambda_{42},...,\lambda_{m2}) &= (0,0,0,...,0), \text{ and }  \\
(\lambda_{22},\lambda_{32},\lambda_{42},...,\lambda_{m2}) &= (0,1,0,...,0).
\end{align*}
By symmetry, the RLCTs at both points are the same, so we consider the first point. 
As $\lambda_{32} - 1$ is non-zero at this point, and hence is a unit in the ring of real analytic functions on a small neighborhood, we may instead consider the ideal 
\[
\mathcal{J}_0 = \langle \lambda_{22}^2 \rangle \cdot \langle \lambda_{32}, \ldots, \lambda_{p2}\rangle.
\]
By Example~\ref{example:linear-space}, we can now compute
\[
\rlct_0(\langle \lambda_{32},\ldots,\lambda_{p2}\rangle; 1) = (p-2,1).
\]
Moreover, using ~\cite[Theorem 7.1]{lin:2014} or the Newton polyhedron method from Section~\ref{subsection:newton-polyhedra}, we infer that
\[
\rlct_0(\langle \lambda_{22}^2\rangle; \lambda_{22}^{p-2}) = ((p-1)/2,1).
\]
As $p>3$, we get that $(p-1)/2 < p-2$, and so the product rule (Fact~\ref{fact:sum-product-rule}(2)) yields
\[
\rlct_0(\mathcal{J}_0;\lambda_{22}^{p-2}) = ((p-1)/2,1).
\]
Finally, using the sum rule (Fact~\ref{fact:sum-product-rule}(1)), the desired minimum of RLCTs is $((p-1)/2,1) + (2,0) + (p-2,0) = ((3p-1)/2,1)$, where the second summand comes from (\ref{eq:k12line5}) and the third summand comes from (\ref{eq:separate}).
\end{proof}

\begin{cor}
  \label{prop:k=2,r=1}
Let $p > 3$. The learning coefficient and its order of the factor analysis model with observations of dimension $p>3$ and $k = 2$ latent factors along the submodel with $r = 1$ latent factor satisfy

  \[
  \ell_{21} = \frac{5p-1}{4} \quad\text{and}\quad m_{21}=1.
  \]
In the case of $2$-dimensional observations, the invariants are $\ell_{21} = 3/2$ and $m_{21} = 1$.
In the case of $3$-dimensional observations, the invariants are $\ell_{21} = 3$ and $m_{21} = 1$. In each case, the genericity assumption on $\Sigma_0$ from the $1$-factor model is that all of its off-diagonal entries be non-zero, see Section~\ref{subsection:sbic}.
\end{cor}

\begin{proof}
This follows immediately from Lemma~\ref{lemma:k=2,r=1} and Theorem~\ref{fact:reduction}.
\end{proof}

\begin{remark}\label{remark:better-gens-r=1}
    We specialize the set of generators from Remark~\ref{rmk:nicer-eqs} to the case $r = 1$. The matrix $\Lambda_{11}\Lambda_{11}^T - L_{11}L_{11}^T$ from (\ref{eq:matrix1}) is of size $1\times1$ here, and hence has no off-diagonal entries. So, we are left with the off-diagonal entries of the following matrix from (\ref{eq:matrix2}):
    \[
\Lambda_{22}\Lambda_{22}^T+L_{21}L_{11}^T\left[ \left( \Lambda_{11}\Lambda_{11}^T \right)^{-1}-\left(L_{11}L_{11}^T\right)^{-1}\right]L_{11}L_{21}^T
    \]
    Note that $\Lambda_{11} > 0$. Also, $L_{11}$ consist of a single entry and that $L_{21}$ is a vector. As before, we assume that $\Sigma_0$ is generic in the sense that all of its off-diagonal entries are non-zero. This is equivalent to saying that $L_{11}\neq 0$ and all entries of $L_{21}$ are non-zero. Using Lemma~\ref{lem:torus}, we may assume that actually $L_{11} = 1$ and $L_{21}$ is a vector consisting of ones. So, the above matrix simplifies to
    \[
    \Lambda_{22}\Lambda_{22}^T + \left(\frac{1}{\Lambda_{11}^2}-1\right)\cdot \mathbf{1},
    \]
    where $\mathbf{1}$ is a $(p-1)\times (p-1)$-matrix of ones. Leaving the $\Lambda_{22}$-coordinates fixed and applying $a = \frac{1}{\Lambda_{11}^2} - 1$ whose inverse is $\Lambda_{11} = \frac{1}{\sqrt{1+a}}$ gives a real analytic isomorphism. It pulls back the ideal generated by the off-diagonal entries of the matrix above to
    \[
    \mathcal{J}_{p,k} = \langle a + \tilde\lambda_i \tilde\lambda_j^T \mid 2 \leq i < j \leq p \rangle,
    \]
    where $\tilde\lambda_i$ is the $i$-th row of $\Lambda_{22}$ which is of length $k-1$. Following Remark~\ref{rmk:nicer-eqs}, we then have
    \[
    \rlct_\Lambda(\mathcal{I}_{p,k}(\mathbf{1});1) = \rlct_{(a,\Lambda_{22})}(\mathcal{J}_{p,k};1) + (p-1,0).
    \]
\end{remark}

\begin{lemma}\label{lemma:r=1}
 
 Let $k \geq 1$ and $p \geq k+2$. Then
\begin{equation*}\label{eq:rlct}
\min_{(a,\Lambda_{22})}\rlct_{(a,\Lambda_{22})}(\cJ_{p+1,k+1};1)= \left(\frac{pk}{2}+1,1\right),
\end{equation*}
where the minimum ranges over all $(a,\Lambda_{22})$ whose image under the transformation in Remark~\ref{remark:better-gens-r=1} is $\Lambda \in \R^{p\times k}$ in LQ-coordinates such that the off-diagonal entries of $\Lambda\Lambda^T$ equal~$1$. In particular, if $\Sigma_0$ is chosen generically from the $1$-factor model then
\[
\min_{\Lambda} \rlct_\Lambda(\mathcal{I}_{p,k}(\Sigma_0);1) = \left( \frac{pk+p-k+1}{2},1\right),
\]
where the minimum ranges over all $\Lambda \in \R^{p\times k}$ such that $\Sigma_0 = \diag(\psi) + \Lambda \Lambda^T$ for some $\psi \in  \R_{ > 0}^p$.

With analogous notation for $p \geq 2$ and $k \in \{p-1, p\}$, 
\begin{equation*}
\min_{(a,\Lambda_{22})}\rlct_{(a,\Lambda_{22})}(\cJ_{p+1,k+1};1)= \left(\frac{p(p-1)}{2},1\right)
\end{equation*}
and
\[
\min_{\Lambda} \rlct_\Lambda(\mathcal{I}_{p,k}(\Sigma_0);1) = \left( \frac{p(p-1)}{2},1\right).
\]
\end{lemma}

\begin{proof}
    We proceed by induction on $k$. The base case $k = 1$ follows directly from Lemma~\ref{lemma:k=2,r=1} and Remark~\ref{remark:better-gens-r=1} because 
    \[
    \min_{(a,\Lambda_{22})} \rlct_{(a,\Lambda_{22})}(\mathcal{J}_{p+1,1+1};1)  =\left(\frac{3(p+1)-1}{2} , 1\right) - ((p+1)-1,0) = \left(\frac{p\cdot 1}{2} + 1,1\right).
    \]
    Now, let $k\geq 2$. For simplicity of notation, we denote $\Lambda_{22} = (\gamma_{ij})$ with $1 \leq i \leq p$ and $1 \leq j \leq k$. We write $\gamma_i$ for the $i$-th row of $\Lambda_{22}$. Changing generators, we may write $\mathcal{J}_{p+1,k+1} = \mathcal{J}_1 + \mathcal{J}_2$, where $\mathcal{J}_1 = \langle a + \gamma_1\gamma_2^T\rangle$ and
    \begin{align}
        \mathcal{J}_2 &= \langle\gamma_i\gamma_j^T - \gamma_1\gamma_2^T \mid 1 \leq i < j \leq p , (i,j) \neq (1,1) \rangle \\
        \label{equation:J_2}&= \langle \gamma_i\gamma_j - \gamma_e \gamma_f \mid 1 \leq i < j \leq p, 1 \leq e < f \leq p\rangle.
    \end{align}
    Now we apply the real analytic map with Jacobian determinant equal to $1$ given by $a' = a + \gamma_i\gamma_j^T$ that leaves the $\gamma_i$-coordinates unchanged. Under this map, $\mathcal{J}_2$ is left unchanged while $\mathcal{J}_1$ is transformed to $\langle a' \rangle$ which adds $(1,0)$ to the $\rlct$ by Fact~\ref{fact:sum-product-rule}(1) and Example~\ref{example:linear-space}.

    So, we only need to consider $\mathcal{J}_2$ which is homogeneous and hence admits its minimal RLCT at $0$. We blow up at this point. The generators of $\mathcal{J}_2$ in (\ref{equation:J_2}) are invariant under the action of the symmetric group and, hence, we just have to consider one chart, say the one corresponding to the variable $\gamma_{11}$. Its Jacobian determinant is $\gamma_{11}^{pk-1}$ and it pulls back $\mathcal{J}_2$ to $\langle \gamma_{11}^2\rangle \cdot \mathcal{K}$, where
    \[
    \mathcal{K} = \langle \gamma_i\gamma_j^T - \gamma_1\gamma_2^T \mid \gamma_{11} = 1, 1 \leq i < j \leq p\rangle.
    \]
    We can compute $\rlct_0(\langle \gamma_{11}^2\rangle;\gamma_{11}^{pk-1}) = (pk/2,1)$, see for instance~\cite[Theorem 7.1]{lin:2014}. By the product rule (Fact~\ref{fact:sum-product-rule}(2)), it remains to compute the minimal RLCT of $\mathcal{K}$ with respect to an amplitude function of $1$.

    The generators of $\mathcal{K}$ are stable under replacing $\Lambda_{22} = (\gamma_{ij})$ by $Q\Lambda_{22}Q^T$, where $Q$ is a orthonormal matrix. As $\gamma_{11} = 1$, the matrix $\Lambda_{22}$ varies over a small set inside the space of matrices of rank at least $1$. So, we can apply LQ-decomposition (Section~\ref{subsection:LQ-decomp}) and assume without changing the RLCT that $\Lambda_{22}$ varies over $\mathcal{L}^{p,k}_{1,+}$. In particular, we may write 
    \[
    \mathcal{K} = \langle \gamma_i\gamma_j^T - \gamma_1\gamma_2^T \mid \gamma_{12} = \ldots = \gamma_{1p} = 0, 1 \leq i < j \leq p\rangle,
    \]
    where $\gamma_{11}$ only takes values greater than $0$, or in other words, is a unit in the ring of analytic functions the ideal $\mathcal{K}$ is considered in.
    So, writing $\tilde\gamma_i$ for the $i$-th row of $\Lambda_{22}$ from which we deleted the first column, $\mathcal{K}$ simplifies to
    \begin{align*}
    \mathcal{K} &= \langle\gamma_{11}(\gamma_{j1} - \gamma_{21}) \mid 3\leq j \leq p\rangle + \langle \gamma_{i} \gamma_j^T- \gamma_{11}\gamma_{21} \mid 2 \leq i < j \leq p\rangle \\
    &= \langle\gamma_{j1} - \gamma_{21} \mid 3\leq j \leq p\rangle + \langle \gamma_{i1}\gamma_{j1} + \tilde\gamma_i\tilde\gamma_j^T - \gamma_{11}\gamma_{21} \mid 2 \leq i < j \leq p\rangle\\
    &= \langle\gamma_{j1} - \gamma_{21} \mid 3\leq j \leq p\rangle + \langle \gamma_{21}^2  - \gamma_{11}\gamma_{21} + \tilde\gamma_i\tilde\gamma_j^T \mid 2 \leq i < j \leq p\rangle.
    \end{align*}
    We apply the substitution $a = \gamma_{21}^2 - \gamma_{21}\gamma_{11}$, sending all variables except $\gamma_{21}$ to itself. The Jacobian determinant of this map is $2\gamma_{21} - \gamma_{11}$. As $\mathcal{K}$ is homogeneous in the entries of the $\gamma_i$, by Fact~\ref{fact:homogeneous} we can restrict these variables to an arbitrarily small neighborhood of the origin.
    
    Assume to the contrary that $2 \gamma_{21} - \gamma_{11} = 0$ for a point on the variety defined by $\mathcal{K}$, which implies in particular that $\gamma_{21} >0$. Then, this point satisfies $\tilde\gamma_i \tilde\gamma_j^T = \gamma_{21}^2$ for all $2\leq i < j \leq p$. As this relation is homogeneous, and hence is invariant under scaling, $\gamma_{21}$ can be assumed arbitrarily small. But this is a contradiction to $\gamma_{11}$ being bounded away from $0$ and $\gamma_{11} = 2\gamma_{21}$.
    
    So, this Jacobian determinant can be assumed to be non-zero, and hence we can ignore it. Transforming $\mathcal{K}$ under this map and applying the substitution $\gamma_{j1} = \gamma_{j1} - \gamma_{21}$ yields
    \[
    \langle \gamma_{j1} \mid 3 \leq j \leq p\rangle + \langle a + \tilde\gamma_i \tilde\gamma_j^T \mid 2 \leq i < j \leq p \rangle,
    \]
    where the second summand is just $\mathcal{J}_{p,k}$. So, by Fact~\ref{fact:sum-product-rule}(1) and Example~\ref{example:linear-space}, the minimal RLCT of $\mathcal{K}$ is just $(p-2,0)$ plus the minimal RLCT of $\mathcal{J}_{p,k}$.

    We distinguish two cases. Suppose first that $k\leq p-2$.
    Then, by the induction hypothesis, the minimal RLCT of $\mathcal{J}_{p,k}$ is $((p-1)(k-1)/2,1)$. As $pk/2 < p-2 + (p-1)(k-1)/2$ for $k \leq p-2$. Taking into account the summand $(1,0)$ from the ideal $\langle a' \rangle$, the product rule (Fact~\ref{fact:sum-product-rule}(2)) implies that the minimal RLCT of $\mathcal{J}_{p+1,k+1}$ equals $(1,0)+ (pk/2,1)= (pk/2 + 1, 1)$ as we claimed. 

    Now, suppose that $k \in \{p-1, p\}$. In this case, by the induction hypothesis, the minimal RLCT of $\mathcal{J}_{p,k}$ is $((p-2)(p-1)/2, 1)$. As $p-2 + (p-2)(p-1)/2 = (p^2-p-2)/2 < pk/2$ for $k \in \{p-1,p\}$, again taking into account the summand $(1,0)$, we infer that the minimal RLCT of $\mathcal{J}_{p+1,k+1}$ equals $((p^2 - p - 2)/2, 1) + (1,0) = ((p-1)p/2,1)$ as asserted.

\end{proof}

\begin{cor}\label{thm:r=1}
The learning coefficient and its order of the factor analysis model with $p$-dimensional observations and $k \leq p-2$ latent factors along the submodel with $r = 1$ latent factor satisfy
  \[
  \ell_{k1} =\frac{pk+3p-k+1}{4}
  \quad\text{and}\quad
  m_{k1} = 1.
  \]
  If $k \in \{p-1,p\}$ then
  \[
  \ell_{k1} = \frac{p(p+1)}{4}  \quad \text{and} \quad m_{k1}=1.
  \]
  In each case, the genericity assumption on $\Sigma_0$ from the $1$-factor model is that all of its off-diagonal entries be non-zero, see Section~\ref{subsection:sbic}.
\end{cor}
\begin{proof}
    This follows immediately from Lemma~\ref{lemma:r=1} and Theorem~\ref{fact:reduction}.
\end{proof}

\section*{Acknowledgements}
We want to thank the anonymous referees for their helpful comments that helped improving the manuscript substantially. 
D. Kosta gratefully acknowledges funding from the Royal Society Dorothy Hodgkin Research Fellowship DHF$\backslash$R1$\backslash$201246. We would like to thank the associated Royal Society Enhancement grant RF$\backslash$ERE$\backslash$231053  that supported  D. Windisch's postdoctoral research, when part of this research was carried out. D. Windisch was also partially supported by the FWO grants G0F5921N (Odysseus) and G023721N, and by the KU Leuven grant iBOF/23/064.

\bibliographystyle{amsalpha}
\bibliography{factoranalysis}
   \end{document}